\documentclass[12pt,reqno,oneside]{amsart}

\usepackage{graphicx}
\usepackage{amssymb}
\usepackage{epstopdf}

\usepackage[a4paper, total={6in, 9in}]{geometry}

\usepackage{booktabs} % for much better looking tables
\usepackage{array} % for better arrays (eg matrices) in maths
\usepackage{paralist} % very flexible & customisable lists (eg. enumerate/itemize, etc.)
\usepackage{verbatim} % adds environment for commenting out blocks of text & for better verbatim
\usepackage{subfig} % make it possible to include more than one captioned figure/table in a single float
\usepackage{tabularx}

\usepackage{amsmath,amsfonts,amsthm,mathrsfs,amssymb,cite}
\usepackage[usenames]{color}

\newtheorem{thm}{Theorem}[section]

\newtheorem{cor}[thm]{Corollary}
\newtheorem{lem}{Lemma}[section]
\newtheorem{prop}{Proposition}[section]
\theoremstyle{definition}
\newtheorem{defn}{Definition}[section]
\newtheorem{asm}{Assumption}

\theoremstyle{remark}

\newtheorem{rem}{Remark}[section]
\numberwithin{equation}{section}

\numberwithin{equation}{section}

\newcounter{saveeqn}
\def\nm{\noalign{\medskip}}
\newcommand{\eqnref}[1]{(\ref {#1})}

\newcommand{\bE}{\mathbf{E}}
\newcommand{\bJ}{\mathbf{J}}
\newcommand{\bF}{\mathbf{F}}

\newcommand{\bG}{\mathbf{G}}
\newcommand{\bH}{\mathbf{H}}

\newcommand{\Bx}{\mathbf{x}}
\newcommand{\By}{\mathbf{y}}

\newcommand{\BU}{\mathbf{U}}

\newcommand{\Gl}{\lambda}

\newcommand{\Gs}{\sigma}

\newcommand{\GL}{\Lambda}

\newcommand{\Acal}{\mathcal{A}}
\newcommand{\Kcal}{\mathcal{K}}
\newcommand{\Lcal}{\mathcal{L}}
\newcommand{\Scal}{\mathcal{S}}

\newcommand{\Mcal}{\mathcal{M}}

\newcommand{\Ocal}{\mathcal{O}}
\newcommand{\Vcal}{\mathcal{V}}

\newcommand{\ds}{\displaystyle}
\newcommand{\la}{\langle}
\newcommand{\ra}{\rangle}

\newcommand{\RR}{\mathbb{R}}

\newcommand{\p}{\partial}

\newcommand{\beq}{\begin{equation}}
\newcommand{\eeq}{\end{equation}}

\DeclareMathAlphabet{\itbf}{OML}{cmm}{b}{it}

\title[On an inverse boundary problem arising in brain imaging]{On an inverse boundary problem arising in brain imaging}

\author{Youjun Deng}
\address{School of Mathematics and Statistics, Central South University, Changsha, Hunan, P. R. China.}
\email{youjundeng@csu.edu.cn, dengyijun\_001@163.com}

\author{Hongyu Liu}
\address{Department of Mathematics, Hong Kong Baptist University, Kowloon, Hong Kong SAR}

\email{hongyu.liuip@gmail.com}

\author{Gunther Uhlmann}
\address{Department of Mathematics, University of Washington, Seattle, WA 98195, USA/ Jockey Club Institute for Advanced Study, HKUST, Hong Kong SAR / Department
of Mathematics and Statistics, University of Helsinki, Finland}
\email{gunther@math.washington.edu}

\date{} % Activate to display a given date or no date (if empty),
\begin{document}
\maketitle

\begin{abstract}

We consider the inverse problem of recovering both an unknown electric current and the surrounding electromagnetic parameters of a medium from boundary measurements. This inverse problem arises in brain imaging. We show that under generic conditions one can recover both the source and the electromagnetic parameters if these
are piecewise constant and the current source is invariant in a fixed direction or a harmonic function.

\medskip

\smallskip


\end{abstract}

\section{Introduction}

\subsection{Mathematical setup}\label{sect:1.1}
We consider the electromagnetic (EM) phenomenon due to an electric source located inside an inhomogeneous medium. Let $B$ be a bounded simply connected open set in $\mathbb{R}^3$ with a $C^{1,1}$-smooth boundary $\partial B$. Throughout we let $\Bx=(\Bx_j)_{j=1}^3\in\mathbb{R}^3$ denote the space variable and $t$ denote the time variable. Let $\mathcal{J}(\Bx, t): \mathbb{R}^3\times \mathbb{R}_+\rightarrow \mathbb{R}^3$ be a causal function such that $\mathcal{J}(\Bx, t)=0$ for $(\Bx, t)\in \mathbb{R}^3\backslash\overline{B}\times\mathbb{R}_+$. It models an electric current source supported inside $B$. We characterize the EM medium by the electric permittivity $\epsilon(\Bx)$, magnetic permeability $\mu(\Bx)$ and electric conductivity $\sigma(\Bx)$. $\epsilon(\Bx)$ and $\mu(\Bx)$ are positive scalar functions that are bounded below and above, and $\sigma(\Bx)$ is a nonnegative scalar function that is bounded above, such that $\epsilon(\Bx)=\epsilon_0$, $\mu(\Bx)=\mu_0$ and $\sigma(\Bx)=0$ for $\Bx\in\mathbb{R}^3\backslash\overline{B}$, where $\epsilon_0$ and $\mu_0$ are positive constants. $\epsilon_0$ and $\mu_0$ signify the EM parameters of the homogeneous background space $\mathbb{R}^3\backslash\overline{B}$. Let $\mathcal{E}$ and $\mathcal{H}$ be both $\mathbb{R}^3$-valued functions, which respectively model the electric and magnetic fields generated by the source current $\mathcal{J}$. We have the following Maxwell system,
\begin{equation}\label{eq:Maxwell2}
\begin{cases}
&\displaystyle{ -\epsilon(\mathbf{x}){\partial_t\mathcal{E}(\mathbf{x}, t)}+\nabla\times\mathcal{H}(\mathbf{x}, t)=\mathcal{J}(\mathbf{x}, t)+\sigma(\mathbf{x})\mathcal{E}(\mathbf{x}, t), }\medskip\\
&\displaystyle{\mu(\mathbf{x}) {\partial_t\mathcal{H}(\mathbf{x},t)}+\nabla\times\mathcal{H}(\mathbf{x}, t)=0\quad (\mathbf{x}, t)\in\mathbb{R}^3\times\mathbb{R}_+},\medskip\\
&\displaystyle{\mathcal{E}(\mathbf{x}, 0)=\mathcal{H}(\mathbf{x}, 0)=0,\quad \mathbf{x}\in\mathbb{R}^3,}.
\end{cases}
\end{equation}
Associated with \eqref{eq:Maxwell2}, we define a boundary measurement operator
\begin{equation}\label{eq:bop1}
\Pi_{\mathcal{J},\epsilon,\mu,\sigma}(\mathbf{x}, t)=(\mathcal{E},\mathcal{H})(\mathbf{x}, t)|_{(\mathbf{x}, t)\in\partial B\times\mathbb{R}_+},
\end{equation}
where $(\mathcal{E},\mathcal{H})$ is the pair of solutions to \eqref{eq:Maxwell2}. In this article, we are concerned with an inverse problem of simultaneously recovering $\mathcal{J}, \epsilon, \mu$ and $\sigma$ by knowledge of $\Pi_{\mathcal{J},\epsilon,\mu,\sigma}$, namely
\begin{equation}\label{eq:ip}
\Pi_{\mathcal{J},\epsilon,\mu,\sigma}(\mathbf{x}, t)\rightarrow (B; \mathcal{J}, \epsilon, \mu, \sigma).
\end{equation}
For the aforementioned inverse boundary problem, we aim to establish sufficient conditions for the following unique recovery result holding true
\begin{equation}\label{eq:uniqueness2}
\Pi_{\mathcal{J}_1,\epsilon_1,\mu_1,\sigma_1}=\Pi_{{\mathcal{J}}_2, \epsilon_2, \mu_2, \sigma_2}\Leftrightarrow (\mathcal{J}_1, \epsilon_1,\mu_1,\sigma_1)=({\mathcal{J}}_2, \epsilon_2, \mu_2, \sigma_2),
\end{equation}
where $(\mathcal{J}_1, \epsilon_1,\mu_1,\sigma_1)$ and $({\mathcal{J}}_2, \epsilon_2, \mu_2, \sigma_2)$ are two sets of admissible EM configurations.

\subsection{Background and motivation}

In the physical setting, $(B; \epsilon,\mu,\sigma)$ denotes an inhomogeneous EM medium that is embedded in a homogeneous background space $(\mathbb{R}^3\backslash\overline{B}; \epsilon_0, \mu_0)$. $\mathcal{J}$ signifies an electric current density that is located inside the body $(B; \epsilon,\mu,\sigma)$. The presence of the source $\mathcal{J}$ generates the EM fields $\mathcal{E}$ and $\mathcal{H}$ that propagate (nondestructively) outside the body $B$. The inverse problem \eqref{eq:ip} is concerned with the inference of the knowledge of the interior of the body $B$, which is a natural wave probing approach that has been widely adopted in many applications.

For the inverse problem \eqref{eq:ip} described above, if one assumes that the medium $(B; \epsilon,\mu, \sigma)$ is known,  but $\mathcal{J}$ is unknown, then the problem is usually referred to as an inverse source problem; see \cite{Bao1,Isa} and the references therein for relevant discussions. This inverse source problem is of particular significance to electroencephalography (EEG) and magnetoencephalography (MEG) (cf. \cite{Coh1,Coh2}), which are two important brain imaging methods. It is known that brain activity induces EM fields and the measurement of the induced EM fields can be used to understand the brain processes, which is actually modelled as the inverse source problem associated with \eqref{eq:ip}. We are aware of some recovery results in the literature associated with EEG and MEG\cite{AB,FGK,FKM}. However, these articles deal with the recovery of the source term, corresponding to the neuronal current inside the brain, without considering the simultaneous recovery of the source and the surrounding medium. The inverse source problem is linear. If one assumes that $\mathcal{J}$ is known, but the medium $(B; \epsilon,\mu, \sigma)$ is unknown, then the corresponding inverse problem \eqref{eq:ip} is usually referred to as an inverse medium problem (cf. \cite{CK,Isa,18,19,SyU,U,Uhl}). We would like to remark that the inverse medium problems considered in the literature are usually active in the sense that the measurement data are generated by sending EM fields from outside of the body and then by measuring the EM responses from outside of the body as well. In this paper, we consider the inverse problem \eqref{eq:ip} by assuming that both the source $\mathcal{J}$ and the surrounding medium $(B;\epsilon,\mu,\sigma)$ are unknown. It is easily verified that the proposed inverse problem is nonlinear. To our best knowledge, the inverse problem is new to the literature and the corresponding study becomes radically more challenging. According to our earlier discussion, it would be of practical interest in EEG and MEG if one intends to infer knowledge on both the brain activity and the brain tissue and then study their correlations by exterior EM measurements. In the current article, we shall be mainly concerned with the theoretical identifiability results, namely \eqref{eq:uniqueness2}. That is, given the measurement data set $\Pi_{\mathcal{J},\epsilon,\mu,\sigma}$, we aim to establish sufficient conditions for $\mathcal{J}$ and $(B;\epsilon,\mu,\sigma)$ such that both of them can be identified. Generically speaking, if the medium parameters are piecewise constants and the source current is invariant along any given fixed direction, then under generic conditions all of them can be recovered. To our best knowledge these identifiability results are new to the literature. 

Another practical scenario that is related to our proposed study is the thermoacoustic and photoacoustic tomography \cite{StU1, StU2, StU3, LiuUhl15}, where one intends to infer knowledge of an inside body by exterior acoustic wave measurement that is generated by an internal source. In particular, in a recent article \cite{LiuUhl15}, the simultaneous recovery of an internal source and the sound speed of the surrounding medium in thermoacoustic and photoacoustic tomography was considered, where the governing PDE is the scalar wave equation. For the proposed inverse problem \eqref{eq:ip} associated with the Maxwell system \eqref{eq:Maxwell2}, one would encounter more complicated analysis and technical difficulties in establishing the corresponding identifiability results than those in \cite{LiuUhl15}.

\subsection{Some remarks}

Before beginning with the mathematical study, several general remarks are in order. The major findings in this paper are contained in Theorems~\ref{th:unique1} and \ref{th:unique2}. Generically speaking, if the source term $\mathcal{J}(\Bx, t)=\mathbf{J}(\Bx)\delta (t)$ with $\mathbf{J}$ invariant along any given direction or a harmonic function, and the medium $(B; \epsilon, \mu, \sigma)$ possesses material parameters being piecewise-constant, then one can recover both the source and the medium. It is remarked that our simultaneous recovery results may not be exclusive, and the unique recovery might hold for other scenarios. The main mathematical arguments to establish the recovery result are first to reduce the Maxwell system \eqref{eq:Maxwell2} from the time regime to the frequency regime by Fourier transform and then to derive integral representations of the electric and magnetic fields, respectively. Next by performing asymptotic analysis in the low frequency regime, we can derive certain integral identities involving the source function and the material parameters, which are coupled together. By inverting those integral identities using harmonic analysis techniques, we obtain the desired unique recovery results.

In Section 2, we present some results concerning the forward Maxwell system \eqref{eq:Maxwell2}, especially the integral representations and the asymptotic expansions of the solutions with respect to the frequency $\omega$. Section 3 is devoted to the unique recovery results.
%The paper is concluded in Section 4 with some relevant discussions.

\vspace*{-.1cm}

\section{Auxiliary results on the forward Maxwell system}

In order to provide a functional analysis framework for the investigation of \eqref{eq:ip} associated with \eqref{eq:Maxwell2}, we first introduce some Sobolev spaces (cf. \cite{LRX,Ned}). We often use the spaces
\[
H_{loc}(\mbox{curl}; X)=\{ \BU|_D\in H(\mbox{curl}; D);\ D\ \ \mbox{is any bounded subdomain of $X$} \}
\]
and
\[
H(\mbox{curl}; D)=\{ \BU\in (L^2(D))^3;\ \nabla\times \BU\in (L^2(D))^3 \}.
\]
We also define the following spaces
\[
H_{loc}(\mbox{div}; X)=\{ \BU|_D\in H(\mbox{div}; D);\ D\ \ \mbox{is any bounded subdomain of $X$} \}
\]
and
\[
H(\mbox{div}; D)=\{ \BU\in (L^2(D))^3;\ \nabla\cdot \BU\in (L^2(D))^3 \}.
\]
Furthermore, for $\beta\in L^\infty(D)$, we denote by $H(\mbox{div}(\beta\cdot); D)$ the function space
\[
H(\mbox{div}(\beta\cdot); D)=\{ \BU|_D\in (L^2(D))^3; \quad \nabla\cdot (\beta \BU)\in (L^2(D))^3 \}
\]
In the sequel, we assume that
\begin{equation}\label{eq:th2}
\mathcal{J}(\Bx, t)=\mathbf{J}(\Bx) \delta (t),
\end{equation}
with $\bJ\in H_{0}(\mbox{div}; B)$, where $H_0(\mbox{div}; B)$ stands for the set of all functions that are in $H(\mbox{div}; B)$ and compactly supported in $B$. We define similarly for space $H_0(\mbox{curl}; B)$. In \eqref{eq:th2}, $\delta$ signifies the delta distribution. It is also assumed that $\epsilon, \mu$ and $\sigma$ belong to $L^\infty(\mathbb{R}^3)$. It is recalled that
 \begin{equation}\label{eq:outbd1}
\epsilon=\epsilon_0,  \quad \mu =\mu_0, \quad\bJ=0 \quad \mbox{and} \quad \Gs=0 \quad \mbox{in} \quad \RR^3\setminus\overline{B}.
\end{equation}
We refer to \cite{Leis1,Leis2} for the well-posedness of the forward Maxwell system \eqref{eq:pss}, and in particular the unique existence of a pair of solutions $(\mathcal{E}, \mathcal{H})\in C(\mathbb{R}_+^0, H_{loc}(\mathrm{curl},\mathbb{R}^3))^2$.

\subsection{The reduced Maxwell system}

We shall make use of the method of Fourier transform to reduce the Maxwell system \eqref{eq:Maxwell2} from the time-domain to the frequency-domain. To that end, we introduce the following temporal Fourier transform $\mathcal{F}_t: L^2(\mathbb{R}_+)^3\rightarrow L^2(\mathbb{R}_+)^3$,
\begin{equation}\label{eq:ft1}
\mathbf{J}(\omega)=\mathcal{F}_t(\mathcal{J}):=\frac{1}{2\pi}\int_0^\infty \mathcal{J}( t) e^{\mathrm{i}\omega t}\ dt, \quad\mathcal{J},\ \ \omega\in\mathbb{R}_+. 
\end{equation}
Throughout, for the Maxwell system \eqref{eq:Maxwell2}, we assume that the temporal Fourier transforms exist of $\mathcal{E}, \partial_t\mathcal{E}, \mathcal{H}$ and $\partial_t \mathcal{H}$. Indeed, we note that in \eqref{eq:Maxwell2}, if $\sigma$ is not identically vanishing, the Maxwell system retains a damping term, and hence the corresponding EM fields decays exponentially as time $t$ goes to infinity; see \cite{Leis2} for the relevant study. However, all of our subsequent results hold as long as the aforementioned Fourier transforms exist. Hence, in order to appeal for a more general study, we only require the existence of the corresponding Fourier transforms. Moreover, we always assume that the EM fields to \eqref{eq:Maxwell2} are outward radiating, and this can be fulfilled by requiring a certain causality condition on the EM configuration $(B; \mathcal{J}, \epsilon, \mu, \sigma)$ (cf. \cite{Leis2}). However, we shall not explore this point in the current article, and from a physical point of view, this is generically reasonable. 

Next, by applying the temporal Fourier transform to \eqref{eq:Maxwell2} and setting
\[
\mathbf{E}(\Bx; \omega)=\mathcal{F}_t(\mathcal{E}(\Bx, t)), \quad \mathbf{H}(\Bx; \omega)=\mathcal{F}_t(\mathcal{H}(\Bx, t)),
\]
we have the following reduced time-harmonic Maxwell system
 \begin{equation}\label{eq:pss}
\left \{
 \begin{array}{ll}
\nabla\times{\bE}-i\omega\mu{\bH}=0  &\mbox{in} \quad \RR^3,\\
\nabla\times{\bH}+i\omega(\epsilon+i\frac{\sigma}{\omega}) {\bE}=\bJ &\mbox{in} \quad \RR^3,
 \end{array}
 \right .
 \end{equation}
 subject to the Silver-M\"{u}ller radiation condition:
\beq\label{eq:radia2}
\lim_{\|\Bx\|\rightarrow\infty} \|\Bx\| \big( \bH \times\hat{\Bx}- \bE\big)=0,
\eeq
where $\hat{\Bx} = \Bx/\|\Bx\|$ for $\mathbf{x}\in\mathbb{R}^3\backslash\{0\}$. The Silver-M\"uller radiation condition characterizes the outgoing nature of the EM waves (cf. \cite{CK,Leis2,Ned}). In what follows, if $\mathbf{E}$ and $\mathbf{H}$ satisfy \eqref{eq:radia2}, they are referred to as radiating fields.

We refer to \cite{LRX} for the unique existence of solutions $\mathbf{E}, \mathbf{H}\in H_{loc}(\mbox{curl}; \mathbb{R}^3)$ to the Maxwell system \eqref{eq:pss}. It is readily seen that $\bE \in H_{loc}(\mbox{curl}; \mathbb{R}^3)\cap H(\mbox{div}(\epsilon\cdot); \RR^3)$ and $\bH \in H_{loc}(\mbox{curl}; \mathbb{R}^3)\cap H(\mbox{div}(\mu\cdot); \RR^3)$.
%\textcolor[rgb]{1.00,0.00,0.00}{For the sake of simplicity, in the sequel, we let $\epsilon(\Bx)$, $\mu(\Bx)$ and $\Gs(\Bx)$ are all in $C^1(B)$, for any $\Bx\in B$. Thus $\bE(\Bx)\in H(\mbox{curl}; B)\cap H(\mbox{div}; B)$ and $\bH(\Bx)\in H(\mbox{curl}; B)\cap H(\mbox{div}; B)$, for any $\Bx\in B$. By Helmholtz decomposition in free space, which is
%\beq\label{eq:helmd01}
%\bF=-\nabla u + \nabla\times \bf{A},
%\eeq
%where $u$ and $\bf{A}$ satisfy
%\beq\label{eq:helmd02}
%u(\Bx)=\frac{1}{4\pi}\int_{\RR^3} \frac{(\nabla\cdot\bF)(\By)}{\|\Bx-\By\|}d\By, \quad \mbox{and} \quad
%\bf{A}(\Bx)=\frac{1}{4\pi}\int_{\RR^3} \frac{(\nabla\times\bF)(\By)}{\|\Bx-\By\|}d\By.
%\eeq
%By replacing $\bF$ in \eqnref{eq:helmd01} by $\bE$ and $\bH$, respectively and using the property that volume potential in \eqnref{eq:helmd02} maps $L^2(\RR^3)$ to $H^2(\RR^3)$, one can easily find that both $\bE$ and $\bH$ belong to $H^1(\RR^3)$.}

With the above preparations, the inverse problem \eqref{eq:ip} can be reformulated in the frequency-domain associated with the reduced Maxwell system \eqref{eq:pss} as
\begin{equation}\label{eq:ip2}
\Pi_{\mathbf{J},\epsilon,\mu,\sigma}(\mathbf{x}, \omega)\rightarrow (B; \mathbf{J}, \epsilon, \mu, \sigma).
\end{equation}
where
\begin{equation}\label{eq:bop2}
\Pi_{\mathbf{J},\epsilon,\mu,\sigma}(\mathbf{x}, \omega)=(\mathbf{E},\mathbf{H})(\mathbf{x}; \omega)|_{(\mathbf{x}, \omega)\in\partial B\times\mathbb{R}_+}.
\end{equation}
We would like to point out that as part of our results on  the inverse problem \eqref{eq:ip2}, if $\mu$ is known, then it is sufficient to  make use of the measurements of the electric field $\mathbf{E}$. Similarly, if $\epsilon$ is known, then it is enough to make  use of the measurements of the magnetic field $\mathbf{H}$; see our main recovery results in Theorem~\ref{th:unique1}. Furthermore, in all cases, we shall actually only make use of the frequency $\omega$ in a neighborhood of the zero frequency.

\subsection{Integral representations}\label{sect:3}

In this subsection, for the subsequent use, we present the integral representations of the electric field $\mathbf{E}$ and magnetic field $\mathbf{H}$ to the Maxwell system \eqref{eq:pss}. Our discussion follows the general treatment in \cite{Ned}.

Define $k_0:=\omega\sqrt{\epsilon_0\mu_0}$ to be the wave number.  Let $ \Gamma_{k_0}$ be the outgoing fundamental solution to the PDO $-(\Delta+k_0^2)$, that is given by
\begin{equation}\label{Gk} \ds \Gamma_{k_0}
(\Bx) = \frac{e^{ik_0 \|\Bx\|}}{4 \pi \|\Bx\|}.
 \end{equation}
For any bounded domain $B\subset \RR^3$, we denote by $\Vcal_{B}^{k_0}: L_{loc}^2(B)^3\rightarrow L_{loc}^2(B)^3$ the volume potential operator defined by
\beq\label{eq:volumept1}
\Vcal_{B}^{k_0}[\Phi](\Bx):=\int_{B}\Gamma_{k_0}(\Bx-\By)\Phi(\By)d\By.
\eeq
We also denote by $\Scal_B^{k_0}: H^{-1/2}(\p B)\rightarrow H^{1}(\RR^3\setminus\p B)$ the single layer potential operator given by
\beq\label{eq:layperpt1}
\Scal_{B}^{k_0}[\phi](\Bx):=\int_{\p B}\Gamma_{k_0}(\Bx-\By)\phi(\By)d s_\By, \quad \Bx\in \RR^3\setminus\p B,
\eeq
and the double layer potential $\mathcal{D}_{B}^{k_0}:H^{1/2}(\p B)\rightarrow H^{1}(\RR^3\setminus\p B)$ given by
$$
\mathcal{D}_{B}^{0}[\varphi](\Bx):=\int_{\p B}\frac{\p}{\p \nu_y}\Gamma_{0}(\Bx-\By)\varphi(\By)d s_\By, \quad \Bx\in \RR^3\setminus\p B,
$$
and $\Kcal_B^{k_0}: H^{1/2}(\p B)\rightarrow H^{1/2}(\p B)$ the Neumann-Poincar\'e operator
\beq\label{eq:layperpt2}
\Kcal_{B}^{k_0}[\phi](\Bx):=\mbox{p.v.}\quad\int_{\p B}\frac{\p\Gamma_{k_0}(\Bx-\By)}{\p \nu_y}\phi(\By)d s_\By,
\eeq
where p.v. stands for the Cauchy principle value.
In \eqref{eq:layperpt2} and also in what follows, $\nu$ signifies the exterior unit normal vector to the boundary of the domain concerned.
It is known that the single layer potential operator $\Scal_B^{k_0}$ and the double layer potential operator $\mathcal{D}_B^{k_0}$ satisfy the trace formulae
\beq \label{eq:trace}
\begin{split}
\frac{\p}{\p\nu}\Scal_B^{k_0}[\phi] \Big|_{\pm} = (\mp \frac{1}{2}I+
(\Kcal_{B}^{k_0})^*)[\phi] &\quad \mbox{on } \p B, \\
\mathcal{D}_{B}^{0}[\varphi](\Bx)\Big|_{\pm}=(\pm \frac{1}{2}I+
\Kcal_{B}^{k_0})[\varphi] &\quad \mbox{on } \p B,
\end{split}
\eeq
where $(\Kcal_{B}^{k_0})^*$ is the adjoint operator of $\Kcal_B^{k_0}$ with respect to the $L^2$ inner product. 

Define a $6\times 6$ matrix operator $\bG$ as follows,
\beq\label{eq:bG}
\bG(\Bx)= \left(
\begin{array}{ll}
(k_0^2I_{3}+\nabla^2)\Gamma_{k_0}(\Bx) & i\omega\mu_0\nabla\times (\Gamma_{k_0}(\Bx)I_3)\\
-i\omega\epsilon_0\nabla\times(\Gamma_{k_0}(\Bx)I_3) & (k_0^2I_{3}+\nabla^2)\Gamma_{k_0}(\Bx)
\end{array}
\right),
\eeq
where $I_n$ is the $n\times n$ identity matrix. Then the solution to \eqnref{eq:pss} and \eqnref{eq:radia2} can be represented by
\beq\label{eq:repre1}
\left(
\begin{array}{l}
\bE\\
\bH
\end{array}
\right)
=\int_{\RR^3}\bG(\cdot-\By)
\left(
\begin{array}{c}
\tilde{\gamma}(\By)\bE(\By)+i/(\omega\epsilon_0) \bJ(\By)\\
\tilde{\mu}(\By)\bH(\By)
\end{array}
\right)
d\By,
\eeq
where and also in what follows,
\begin{equation}\label{eq:defnew}
\mbox{$\tilde{\gamma}:=(\epsilon+i\Gs/\omega-\epsilon_0)/\epsilon_0$\ \ \ and\ \ \ $\tilde\mu:=(\mu-\mu_0)/\mu_0$}.
\end{equation}
It can be readily seen from \eqnref{eq:repre1} that $\bE(\Bx)= \Ocal(\|\Bx\|^{-1})$, $\bH(\Bx)= \Ocal(\|\Bx\|^{-1})$ as $\|\Bx\|\rightarrow +\infty$.

One can rewrite \eqnref{eq:repre1} as
\begin{equation}\label{eq:ll}
\begin{split}
\bE = & (k_0^2I_{3}+D^2)\int_{\RR^3}\Gamma_{k_0}(\Bx-\By)(\tilde{\gamma}\bE+i/(\omega\epsilon_0) \bJ)(\By)d\By\\
&+i\omega\mu_0\nabla\times \int_{\RR^3}\Gamma_{k_0}(\Bx-\By)\tilde{\mu}(\By)\bH(\By)d\By,\\
\bH = & -i\omega\epsilon_0\nabla\times\int_{\RR^3}\Gamma_{k_0}(\Bx-\By)(\tilde{\gamma}\bE+i/(\omega\epsilon_0) \bJ)(\By)d\By \\
&+(k_0^2I_{3}+D^2)\int_{\RR^3}\Gamma_{k_0}(\Bx-\By)\tilde{\mu}(\By)\bH(\By)d\By.
\end{split}
\end{equation}
By virtue of \eqref{eq:ll}, along with straightforward calculations, one can further show that
\beq\label{eq:repre2}
\left(
\begin{array}{c}
\omega\bE\\
\bH
\end{array}
\right)
=\Mcal_{B}^{k_0}
\left(
\begin{array}{c}
\tilde\gamma\omega\bE\\
\tilde\mu\bH
\end{array}
\right)
+
\left(
\begin{array}{c}
i/\epsilon_0(k_0^2I_{3}+D^2)\Vcal_{B}^{k_0}[\bJ]\\
\nabla\times\Vcal_{B}^{k_0}[\bJ]
\end{array}
\right) \quad \mbox{in} \quad \RR^3,
\eeq
where the operator $\Mcal_{B}^{k_0}$ is defined by
\beq\label{eq:opM}
\Mcal_{B}^{k_0}:=
\left(
\begin{array}{cc}
(k_0^2I_{3}+D^2)\Vcal_{B}^{k_0} & i\omega^2\mu_0\nabla\times\Vcal_{B}^{k_0}\\
-i\epsilon_0\nabla\times\Vcal_{B}^{k_0} & (k_0^2I_{3}+D^2)\Vcal_{B}^{k_0}
\end{array}
\right).
\eeq
We mention that the notation $D^2\Vcal_B^{k_0}$ appearing in \eqnref{eq:ll}, \eqnref{eq:repre2} and \eqnref{eq:opM} stands for $\nabla(\nabla\cdot\Vcal_B^{k_0})$, and we shall also make use of such a notation in the subsequent analysis.

\subsection{Asymptotic expansions}

By using the integral representation \eqref{eq:repre2}, we next derive the asymptotic expansions of $\mathbf{E}$ and $\mathbf{H}$ as $\omega\rightarrow +0$. To that end, we first derive an important lemma. In the following, if $k_0=0$, we formally set $\Gamma_{k_0}$ introduced in \eqref{Gk} to be $\Gamma_0=1/(4\pi\|\Bx\|)$, and the other integral operators introduced in the previous subsection can also be formally defined when $k_0=0$.

\begin{lem}\label{le:opeig01}
The operator $D^2\Vcal_B^{0}$ is semi-negative definite which maps $\rm{GH}_{0}(\mathrm{div}; B)$ to $H^{1}(B)^3$, where $\rm{GH}_{0}(\mathrm{div}; B)$ is defined by
\beq\label{eq:lerevisegr01}
\rm{GH}_{0}(\mathrm{div}; B):=\left\{\Phi\in H_0(\rm{div}; B); \nabla\times\Phi=0\right\}.
\eeq
Furthermore,
the only possible eigenvalue of $D^2\Vcal_B^{0}$ in $H_{0}(\mathrm{div}; B)$ is $-1$ and the corresponding eigenfunction $\Phi$ is in $H^1(B)^3$ verifying
\beq\label{eq:opeig01}
D^2\Vcal_B^0[\Phi]=-\Phi, \quad \nabla\cdot\Phi\neq 0 \quad \mbox{and} \quad \nabla\times \Phi=0.
\eeq
If $\Phi\in H_{0}(\mathrm{div}; B)$ and $\nabla\cdot\Phi=0$ then
\beq\label{eq:opeig02}
D^2\Vcal_B^0[\Phi]=0.
\eeq
\end{lem}
\begin{proof}
Let $\Phi\in H_0(\mathrm{div}; B)$. Using the identity $\p \Gamma_{k_0}(\Bx-\By)/\p \Bx_i = -\p \Gamma_{k_0}(\Bx-\By)/\p \By_i$ and integration by parts, one has
\beq\label{eq:ibp01}
\begin{split}
& D^2\Vcal_B^0[\Phi](\Bx)=\nabla\nabla\cdot\int_B \Gamma_0(\Bx-\By)\Phi(\By)d\By\\
=& \nabla\int_B \Gamma_0(\Bx-\By)(\nabla\cdot \Phi)(\By)d\By=\nabla \Vcal_B^0[\nabla\cdot \Phi].
\end{split}
\eeq
By the fact that $\Vcal_B^0$ maps $L^2(B)$ to $H^2(B)$ (cf. \cite{CK}), one easily sees that $D^2\Vcal_B^0[\Phi]\in H^1(B)^3$.
If $\Phi\in H_{0}(\mbox{div}; B)$ and $\nabla\cdot\Phi=0$ then by using \eqnref{eq:ibp01}, one can prove \eqnref{eq:opeig02}.

Next for $\Phi\in \rm{GH}_{0}(\mathrm{div}; B)$ and $\nabla\cdot\Phi\neq 0$, by integration by parts, one can show
\begin{align*}
\la D^2\Vcal_B^0[\Phi], \Phi\ra&=\int_B D^2\Vcal_B^0[\Phi]\cdot \Phi=\int_B (\nabla\times\nabla\times\Vcal_B^0[\Phi]-\Phi)\cdot \Phi\\
&=\int_B (\nabla\times\Vcal_B^0[\Phi])\cdot (\nabla\times\Phi)-\int_B \Phi\cdot \Phi\\
&=-\int_B \|\Phi\|^2\leq 0,
\end{align*}
which shows that $D^2\Vcal_B^0$ is semi-negative definite on $\rm{GH}_{0}(\mathrm{div}; B)$.

Finally, suppose that $\Gl\neq 0$ is a possible eigenvalue of $D^2\Vcal_B^0$ in $H_0(\mbox{div}; B)$. We have
$$
D^2\Vcal_B^0[\Phi]=\Gl \Phi, \quad \Phi\in H_0(\mbox{div}; B),
$$
which implies that $\Phi$ is also in $H_0(\mbox{curl}; B)$. By taking respectively curl and divergence of both sides of the above equation, one obtains
$$
0=\Gl \nabla\times\Phi, \quad -\nabla\cdot \Phi=\Gl \nabla \cdot \Phi,
$$
which proves \eqnref{eq:opeig01}.

The proof is complete.
\end{proof}
%\begin{rem}
%Lemma \ref{le:opeig01} actually states that the operator $D^2\Vcal_B^0$ maps any vector-valued function $\Phi\in H_0(\mbox{div}; B)$ to a curl-free vector-valued function that is in $H^1(B)^3$. By \eqnref{eq:helmd1} one has
%$$
%D^2\Vcal_B^0[\Phi]=\nabla u.
%$$
%In fact, by integral by parts one has
%$$
%D^2\Vcal_B^0[\Phi]=\nabla\Vcal_B^0[\nabla\cdot\Phi].
%$$
%\end{rem}

We proceed with the asymptotic analysis as $\omega\rightarrow+0$.

\begin{prop}\label{prop:01}
From \eqnref{eq:repre2}, one can show that
\beq\label{eq:opeqB01}
\Acal_{B}^{k_0}\left(
\begin{array}{c}
\omega\bE\\
\bH
\end{array}
\right)
=
\left(
\begin{array}{c}
i/\epsilon_0(k_0^2I_{3}+D^2)\Vcal_{B}^{k_0}[\bJ]\\
\nabla\times\Vcal_{B}^{k_0}[\bJ]
\end{array}
\right) \quad \mbox{in} \quad \RR^3,
\eeq
where $\Acal_B^{k_0}$ is defined by
\beq\label{eq:opA}
\Acal_{B}^{k_0}:= I_6 -
\Mcal_{B}^{k_0}\left(
\begin{array}{cc}
\tilde{\gamma} & 0\\
0 & \tilde\mu
\end{array}
\right).
\eeq
For $\omega\in\mathbb{R}_+$ sufficiently small, we have the following asymptotic expansions of $\Acal_B^{k_0}$ with respect to $\omega$,
\beq\label{eq:expanM01}
\Acal_B^{k_0}=I_6- \Mcal_B^0\left(
\begin{array}{cc}
\tilde{\gamma} & 0\\
0 & \tilde\mu
\end{array}
\right) + \omega^2\mathcal{R}_B^{k_0},
\eeq
where $\Mcal_B^0$ is defined by
\beq\label{eqopM0}
\Mcal_B^{0}:=\left(
\begin{array}{cc}
D^2\Vcal_{B}^{0} & 0\\
-i\epsilon_0\nabla\times\Vcal_{B}^{0} & D^2\Vcal_{B}^{0}
\end{array}
\right),
\eeq
and $\mathcal{R}_B^{k_0}$ is a remainder term which is a bounded operator on
$H^2_{loc}(\RR^3)^3\times H^2_{loc}(\RR^3)^3$.
\end{prop}

\begin{proof}
The proposition can be proved by straightforward calculations using \eqref{eq:repre2} and \eqnref{eq:opM}.
\end{proof}

In what follows, we introduce the following Sobolev space for $s\in\mathbb{R}$ and $|s|\leq 1/2$,
\begin{equation}\label{eq:zeros}
H_0^s(\partial B)=\{u\in H^s(\partial B); \ \int_{\partial B} u\ ds=0\}.
\end{equation}
Similarly we define $L_0^2(\p B)$ to be the space of functions in $L^2$ and has zero average on the boundary.
For $\varepsilon\in L^\infty(B)$ and $\varepsilon>\alpha_0\in\mathbb{R}_+$, we define
$$\GL_{\varepsilon}: H_0^{-1/2}(\p B)\rightarrow H_0^{1/2}(\p B)$$
 to be the Neumann-to-Dirichlet map such that $\GL_{\varepsilon}[\varepsilon\p u/{\p\nu}|_{\p B}]=u|_{\p B}$, where $u\in H^1(B)$ is the solution to
 \[
\nabla\cdot\varepsilon\nabla u=0\quad \mbox{in}\ \ B\ \ \mbox{and}\ \ \int_{\p B}u\ ds=0.
\]
It is remarked that $\Lambda_\varepsilon$ is invertible. Define $N_{\varepsilon}(\Bx,\By)$ to be the Neumann function that satisfies
\beq \label{revn}
 \ \left \{
\begin{array}{l}
\ds \nabla \cdot \varepsilon \nabla N_\varepsilon(\cdot, \By) = -\delta_{\By}(\cdot) \quad
\mbox{in } B , \\
\nm \ds \nu\cdot\varepsilon \nabla N_\varepsilon(\cdot, \By)
\big|_{\p B} = -\frac{1}{|\partial B |} , \quad \ds \int_{\p B}
N_\varepsilon(\Bx,\By) \, d s_\Bx =0 \quad \mbox{for} \quad \By\in B.
\end{array}
\right .
\eeq
We shall also need the following assumption on the permittivity $\epsilon$ and conductivity $\Gs$.
\begin{asm}\label{asm:01}
Assume that ($\epsilon$, $\Gs$) satisfies one of the following two conditions:
\begin{itemize}
  \item[(i)] $\Gs=c\epsilon$ in $B$, where $c\geq 0$ is a constant;
  \item[(ii)] $\epsilon$ and $\Gs$ are piecewise constants in $\mathbb{R}^3$ in the following sense. Set $\Sigma_0:=B$ and assume that $\Sigma_j$, $j=1, 2, \ldots, N$, are Lipschitz subdomains of $\Sigma_0$ such that $\Sigma_{j}\Subset\Sigma_{j-1}$ and $\Sigma_{j-1}\backslash \overline{\Sigma_{j}}$ is connected for  $1\leq j\leq N$. Set $\epsilon^{(0)}=\epsilon_0$, $\sigma^{(0)}=0$ and let $\epsilon^{(j)}$ and $\sigma^{(j)}$ be constants for $j=1,2,\ldots, N+1$. The medium parameters are given as follows,
\beq\label{eq:multilayered}
\begin{split}
&\epsilon=\epsilon^{(0)}\chi(\mathbb{R}^3\backslash \overline{\Sigma_0})+\sum_{j=1}^N \epsilon^{(j)}\chi(\Sigma_{j-1}\backslash\overline{\Sigma_j})+\epsilon^{(N+1)}\chi(\Sigma_{N}),\\
&\sigma=\sigma^{(0)}\chi(\mathbb{R}^3\backslash \overline{\Sigma_0})+\sum_{j=1}^N \sigma^{(j)}\chi(\Sigma_{j-1}\backslash\overline{\Sigma_j})+\sigma^{(N+1)}\chi(\Sigma_{N}).
\end{split}
\eeq
In \eqref{eq:multilayered} and also in what follows, $\chi$ denotes the characteristic function.
\end{itemize}
\end{asm}

With the above preparations, we next show the following important lemma.
\begin{lem}\label{le:01}
Let $\epsilon$, $\mu$ and $\sigma$ be those described in Section~\ref{sect:1.1}. Suppose further that $\Gs$ and $\epsilon$ satisfy Assumption \ref{asm:01}.
Then there exists $\omega_0\in\mathbb{R}_+$ such that for any $\omega\leq \omega_0$, \eqnref{eq:repre2} is uniquely solvable for $(\bE, \bH)\in H_{loc}(\mathrm{curl}; \RR^3)\times H_{loc}(\mathrm{curl}; \RR^3)$ and radiating at infinity. More specifically,
\beq\label{eq:opeqB02}
\left(
\begin{array}{c}
\omega\bE\\
\bH
\end{array}
\right)
=
(\Acal_{B}^{0})^{-1}\left(
\begin{array}{c}
i/\epsilon_0D^2\Vcal_{B}^{0}[\bJ]\\
\nabla\times\Vcal_{B}^{0}[\bJ]
\end{array}\right)
+\Ocal(\omega^2)
\quad \mbox{in} \quad B,
\eeq
where $(\Acal_{B}^{0})^{-1}$ is given by
\beq\label{eq:invA}
\begin{split}
&(\Acal_{B}^{0})^{-1}=\\
&\left(
\begin{array}{cc}
(I_3- D^2\Vcal_{B}^{0}\tilde\gamma)^{-1} & 0\\
-i\epsilon_0(I_3- D^2\Vcal_B^0\tilde\mu)^{-1}\nabla\times\Vcal_{B}^{0}\tilde\gamma(I_3- D^2\Vcal_B^0\tilde\gamma)^{-1} & (I_3- D^2 \Vcal_{B}^{0}\tilde\mu)^{-1}
\end{array}
\right).
\end{split}
\eeq
\end{lem}
\begin{proof}
By direct computations, we first have the following asymptotic expansion for \eqnref{eq:opeqB01} (see \eqnref{eq:expanA2} in what follows for the higher order expansion),
\begin{equation}\label{eq:lll1}
\Acal_{B}^{k_0}= \Acal_{B}^{0}+\Ocal(\omega^2),
\end{equation}
where $\Acal_B^0$ stands for $\Acal_B^{k_0}$ with $k_0$ formally replaced by $0$.
{In view of \eqnref{eq:opeqB01} and \eqref{eq:lll1}, it suffices to show that $\Acal_{B}^{0}$ is invertible and its inverse is given by \eqref{eq:invA}.}
By \eqnref{eq:expanM01}, $\Acal_B^0$ can be represented by
$$\Acal_B^0=I_6- \Mcal_B^0\left(
\begin{array}{cc}
\tilde{\gamma} & 0\\
0 & \tilde\mu
\end{array}
\right).$$
First, we consider the following equation
\beq\label{eq:invert01}
(I_3-D^2\Vcal_B^0\tilde\gamma)[\omega\bE]=i/\epsilon_0 D^2\Vcal_B^0[\bJ] \quad \mbox{in} \ \RR^3
\eeq
with $\bE$ radiating at infinity. One has $\bE\in H_{loc}(\mbox{curl}; \RR^3)$ and satisfies
$$\nabla\times \bE=0, \quad \nabla\cdot ((1+\tilde\gamma)\omega\bE)=i/\epsilon_0\chi(B)\nabla\cdot\bJ \quad \mbox{in} \quad \RR^3.$$
Hence, there exists $u\in H_{loc}^1(\RR^3)$ such that $\bE=\nabla u$ and
\beq\label{eq:forho01}
\nabla\cdot((1+\tilde\gamma)\nabla u)=i\omega^{-1}\epsilon_0^{-1}\chi(B)\nabla\cdot\bJ \quad \mbox{in} \ \RR^3,
\eeq
where $u$ is radiating at infinity.

By the condition (i) in Assumption~\ref{asm:01} that $\Gs=c\epsilon$, one has $1+\tilde\gamma=(1+ci/\omega)\epsilon\epsilon_0^{-1}$
or equivalently,
\beq\label{eq:new01}
\nabla\cdot(\epsilon\nabla u)=c_1(\omega)\chi(B)\nabla\cdot\bJ \quad \mbox{in} \ \RR^3,
\eeq
where $c_1(\omega):=i(\omega+ci)^{-1}$ and $u$ is radiating at infinity. It is readily seen that the existence and uniqueness of a solution to \eqnref{eq:invert01} in $H_{loc}(\mbox{curl}; \RR^3)$ is equivalent to the existence and uniqueness of a solution to \eqnref{eq:new01} in $H_{loc}^1(\RR^3)$. The uniqueness of the solution to \eqnref{eq:new01} is known. In fact, it is equivalent to showing the unique trivial solution of the following equation
$$
\left\{
\begin{array}{ll}
\nabla \cdot\epsilon \nabla u=0 & \mbox{in }  B, \\
\Delta u =0 & \mbox{in } \RR^3 \setminus \overline{B}, \\
\ds \frac{\p u}{\p \nu} \Big|_+=\epsilon \frac{\p u}{\p \nu} \Big|_- & \mbox{on }  \p B,\\
u|_+=u|_- & \mbox{on } \partial B,\\
u(x)=O(\|x\|^{-1}) & \mbox{as } \|x\| \rightarrow \infty,
\end{array}
\right.
$$
which can be found in, e.g., \cite{HK07:book}. For the existence of a solution to \eqnref{eq:new01}, we seek the solution in the following form
 \beq\label{solrep}
 u(\Bx) =
 \begin{cases}
& \displaystyle{\int_{\p B} N_{\epsilon}(\Bx, \By) \psi(\By)ds_\By}\\
& \displaystyle{- c_1(\omega)\int_B N_{\epsilon}(\Bx,\By)(\nabla\cdot\bJ)(\By)d\By} +C,\quad   \Bx \in B,\medskip\\
 & \Scal_B^0 [\phi](\Bx), \qquad  \Bx \in \RR^3 \setminus B,
 \end{cases}
 \eeq
where $(\phi, \psi)\in H^{-1/2}(\p B)\times H_0^{-1/2}(\p B)$ and $C$ is a constant that satisfies
$$
C= \frac{1}{|\p B|}\int_{\p B} \Scal_B^0[\phi](\Bx)ds_\Bx.
$$
By using the transmission conditions, which are implied in \eqnref{eq:new01} and \eqnref{solrep}, $\phi$ and $\psi$ satisfy the following system of integral equations on $\partial B$,
 \beq\label{inteqn}
 \left \{
 \begin{array}{l}
  \ds -\Scal_B^0 [\phi] + \frac{1}{|\p B|} \int_{\p B} \Scal_B^0 [\phi] + \GL_{\epsilon}[\psi] = c_1(\omega)\int_B N_{\epsilon}(\cdot,\By)(\nabla\cdot\bJ)(\By)d\By,\\
 \nm
 \ds -(\frac{1}{2} I + (\Kcal_B^0)^*)[\phi] + \psi = 0.
 \end{array} 
 \right.
 \eeq
The unique solvability of \eqnref{inteqn} in $H^{-1/2}(\p B)\times H_0^{-1/2}(\p B)$ can be found in \cite{ADKL14}, noting that
$\int_B N_{\epsilon}(\cdot,\By)(\nabla\cdot\bJ)(\By)d\By$ is in $H_0^{1/2}(\p B)$. We have proved the unique solvability of \eqnref{eq:new01}, which in turn implies the unique solvability of \eqnref{eq:invert01}.

We proceed to consider the other case that ($\epsilon$, $\Gs$) fulfils (ii) in Assumption~\ref{asm:01}, namely, $\epsilon$ and $\sigma$ are of the form specified in \eqref{eq:multilayered}. Let
$$u=u^R+ iu^I,$$
where $u^R$ and $u^I$ represent the real and imaginary parts of $u$, respectively. Then \eqnref{eq:forho01} can be expanded to the following equations
\beq\label{eq:holap01}
\left\{
\begin{array}{ll}
\epsilon^{(j)}\Delta u^R-\Gs^{(j)}\omega^{-1}\Delta u^I=0, &\mbox{in} \quad \Sigma_{j-1}\setminus\overline\Sigma_{j},\\
\epsilon^{(j)}\Delta u^I+\Gs^{(j)}\omega^{-1}\Delta u^R=\omega^{-1}\nabla\cdot \bJ, &\mbox{in} \quad \Sigma_{j-1}\setminus\overline\Sigma_{j},\\
\epsilon^{(N+1)}\Delta u^R-\Gs^{(N+1)}\omega^{-1}\Delta u^I=0, &\mbox{in} \quad\Sigma_N,\\
\epsilon^{(N+1)}\Delta u^I+\Gs^{(N+1)}\omega^{-1}\Delta u^R=\omega^{-1}\nabla\cdot \bJ. &\mbox{in} \quad \Sigma_N,
\end{array}
\right.
\eeq
where $j=1,2,\ldots, N$.
By \eqnref{eq:holap01}, one can then obtain that
\beq\label{eq:holap02}
\Delta u=\left(\sum_{j=1}^N c_j\chi(\Sigma_{j-1}\backslash\overline{\Sigma_j})+c_{N+1}\chi(\Sigma_{N})\right)\nabla\cdot\bJ
\eeq
holds in $\RR^3$, where
$$
c_j(\omega)=\frac{\Gs^{(j)}+i\epsilon^{(j)}\omega}{(\epsilon^{(j)})^2\omega^2+(\Gs^{(j)})^2}, \quad j=1,2, \ldots, N+1.
$$
Therefore, by \eqref{eq:holap02} and the fact that $u$ is radiating at infinity, it is straightforward to show that \eqnref{eq:holap02} has a unique solution.

Next, we consider the following equation
\beq\label{eq:invert02}
(I_3-D^2\Vcal_B^0\tilde\mu)[\bH]=\nabla\times \Vcal_B^0[\bJ] \quad \mbox{in} \ \RR^3
\eeq
where $\bH$ is radiating at infinity, then one has $\bH\in H_{loc}(\mbox{curl}; \RR^3)$ and
$$\nabla\times( \bH-\nabla\times\Vcal_B^0[\bJ])=0, \quad \nabla\cdot ((1+\tilde\mu)\bH)=0, \quad \mbox{in} \quad \RR^3.$$
Noting that $1+\tilde\mu=\mu$, there exists $u\in H_{loc}^1(\RR^3)$ such that $\bH=\nabla u +\nabla\times\Vcal_B^0[\bJ]$ and
\beq\label{eq:new02}
\nabla\cdot \mu\nabla u=-\nabla\cdot(\mu\nabla\times\Vcal_B^0[\bJ]) \quad \mbox{in} \ \RR^3.
\eeq
Since $\nabla\cdot(\mu\nabla\times\Vcal_B^0[\bJ])=0$ in $\RR^3\setminus \overline{B}$, one can prove similarly that \eqnref{eq:new02} has a unique solution, and thus \eqnref{eq:invert02} has a unique solution.

Now we consider the following more general case. Suppose $\mathbf{M}\in H_{loc}(\mbox{curl}; \RR^3)\cap H_{loc}(\mbox{div}; \RR^3)$ and $\varepsilon\in L^{\infty}(\RR^3)$ and $\varepsilon=0$ in $\RR^3\setminus \overline{B}$. Let $\bF$ be in $H_{loc}(\mbox{curl}; \mathbb{R}^3)\cap H(\mbox{div}((1+\varepsilon)\cdot); \RR^3)$ and radiating at ininity. We prove that the following equation is uniquely solvable
\beq\label{eq:gen01}
(I_3-D^2\Vcal_B^0\varepsilon)[\bF]=\bf M.
\eeq
In fact, by taking curl and divergence of both sides of \eqnref{eq:gen01}, respectively, one has
\beq\label{eq:gen02}
\nabla\times (\bF-{\bf M)}=0, \quad \nabla\cdot (1+\varepsilon)\bF=\nabla\cdot {\bf M}.
\eeq
Then by \eqnref{eq:gen02} there exists $u\in H_{loc}^1(\RR^3)$ such that $\bF-{\bf M}=\nabla u$ and
\beq\label{eq:gen03}
\nabla\cdot (1+\varepsilon)\nabla u = -\nabla\cdot (\varepsilon {\bf M}),
\eeq
and $u$ radiating at infinity. Noting that $\nabla \cdot (\varepsilon {\bf M})=0$ in $\RR^3\setminus \overline{B}$, one can prove the unique solvability of \eqnref{eq:gen03} by using exactly the same method as that for \eqnref{eq:new01} and \eqnref{eq:new02}.

To sum up, we can solve \eqnref{eq:opeqB01} as follows. From the first equation in \eqnref{eq:opeqB01} one can uniquely obtain
$$\omega\bE=i/\epsilon_0(I_3-D^2\Vcal_B^0\tilde\gamma)^{-1}D^2\Vcal_B^0[\bJ]+\Ocal(\omega^2),$$
then by using the second equation, that is,
$$
(I_3- D^2\Vcal_B^0\tilde\mu)[\bH]=-i\epsilon_0\nabla\times\Vcal_B^0[\tilde\gamma\omega\bE]+\nabla\times\Vcal_B^0[\bJ]+\Ocal(\omega^2),
$$
one can uniquely obtain that
$$\bH=-i\epsilon_0(I_3-D^2\Vcal_B^0\tilde\gamma)^{-1}\nabla\times\Vcal_B^0[\tilde\gamma\omega\bE]+ (I_3-D^2\Vcal_B^0\tilde\gamma)^{-1}\nabla\times\Vcal_B^0[\bJ]+\Ocal(\omega^2).$$
Using the above facts, one can directly verify that
$(\Acal_{B}^{0})^{-1}$ is given in \eqref{eq:invA}. Finally, \eqref{eq:opeqB02} can be shown by straightforward asymptotic expansion with respect to $\omega$.

The proof is complete.
\end{proof}

In \eqref{eq:opeqB02}, we only derived the leading-order expansion of $[\omega\mathbf{E},\mathbf{H}]$. In some cases in what follows, we shall also need the higher-order expansion of $[\omega\mathbf{E},\mathbf{H}]$. Next, we present the corresponding high-order expansions. Since the proof is completely similar to that of Lemma~\ref{le:01}, we only sketch it.  Suppose that $\omega$ is sufficiently small and $\Bx\in B$, then there holds the following expansion for $\Mcal_B^{k_0}$ in \eqnref{eq:opM},
\beq\label{eq:expanA}
\Mcal_B^{k_0}=\Mcal_B^0+k_0^2\Mcal_{B,1}+\Ocal(\omega^3),
\eeq
where $\Mcal_B^0$ is defined in \eqnref{eqopM0} and $\Mcal_{B,1}$ is defined by
$$
\Mcal_{B,1}:=\left(
\begin{array}{cc}
\Vcal_B^0+D^2\Lcal_{B} & i/\epsilon_0\nabla\times\Vcal_{B}^{0}\\
-i\epsilon_0\nabla\times\Lcal_{B} & \Vcal_B^0+D^2\Lcal_{B}
\end{array}
\right),
$$
with $\Lcal_B$ given by
\beq\label{eq:L}
\Lcal_B[\Phi]:=-\frac{1}{4\pi}\int_B \|\Bx-\By\|\Phi(\By)d\By.
\eeq
Then by definition \eqnref{eq:opA} we have the following expansion for $\Acal_B^{k_0}$,
\beq\label{eq:expanA2}
\Acal_B^{k_0}=\Acal_{B}^0-k_0^2\Mcal_{B,1}\Upsilon+\Ocal(\omega^3),
\eeq
where $\Upsilon$ is defined by
\beq\label{eq:Upsi1}
\Upsilon:=\left(
\begin{array}{cc}
\tilde{\gamma} & 0\\
0 & \tilde\mu
\end{array}
\right),
\eeq
and
$$\Acal_B^0:=I_6-\Mcal_B^0\Upsilon.$$
With the above computations, we obtain
\begin{thm}\label{th:solexp1}
Suppose that $\epsilon, \mu$ and $\sigma$ are as in Lemma~\ref{le:01}, and in particular note that $\epsilon$ and $\Gs$ satisfy Assumption \ref{asm:01}. Then there exists $\omega_0\in\mathbb{R}_+$ such that for $\omega\leq \omega_0$, the solution
to \eqnref{eq:pss} and \eqnref{eq:radia2} admits the following asymptotic expansion in $B$,
\begin{align}
\left(
\begin{array}{c}
\omega\bE\\
\bH
\end{array}
\right)
=&(\Acal_B^0)^{-1}
\left(
\begin{array}{c}
i/\epsilon_0D^2\Vcal_B^0[\bJ]\\
\nabla\times\Vcal_B^0[\bJ]
\end{array}
\right)+
k_0^2(\Acal_B^0)^{-1}
\left(
\begin{array}{c}
i/\epsilon_0(\Vcal_B^0+D^2\Lcal_B)[\bJ]\\
\nabla\times\Lcal_B[\bJ]
\end{array}
\right)\nonumber\\
&+k_0^2(\Acal_B^0)^{-1}\Mcal_{B,1}\Upsilon(\Acal_B^0)^{-1}
\left(
\begin{array}{c}
i/\epsilon_0D^2\Vcal_B^0[\bJ]\\
\nabla\times\Vcal_B^0[\bJ]
\end{array}
\right)+\Ocal(\omega^3).\label{eq:solexp1}
\end{align}
\end{thm}
\begin{cor}\label{co:01}
Suppose that $\epsilon, \mu$ and $\sigma$ are those in Theorem~\ref{th:solexp1}, and particularly $\epsilon$ and $\Gs$ satisfy Assumption \ref{asm:01}. Then there exists $\omega_0\in\mathbb{R}_+$ such that for $\omega\leq \omega_0$, the solution
to \eqnref{eq:pss} and \eqnref{eq:radia2} admits the following asymptotic expansions in $B$
\beq\label{eq:fist01}
\omega\bE=i/\epsilon_0(I_3- D^2\Vcal_B^0\tilde\gamma)^{-1}D^2\Vcal_B^0[\bJ]+\Ocal(\omega^2),
\eeq
and
\beq\label{eq:fist02}
\bH=-i\omega\epsilon_0(I_3- D^2\Vcal_B^0\tilde\mu)^{-1}\nabla\times\Vcal_B^0[\bE]+(I_3-D^2\Vcal_B^0\tilde\mu)^{-1}\nabla\times\Vcal_B^0[\bJ]+\Ocal(\omega^2).
\eeq
Furthermore, if $\nabla\cdot \bJ=0$, then we have the following asymptotic expansions in $B$,
  \begin{align}
  \bE=& i \omega\mu_0(I_3-D^2\Vcal_B^0\tilde\gamma)^{-1}\nabla\times\Vcal_B^0\tilde\mu(I_3-D^2\Vcal_B^0\tilde\mu)^{-1}\nabla\times\Vcal_B^0[\bJ]\nonumber\\
  &+i\omega\mu_0(I_3-D^2\Vcal_B^0\tilde\gamma)^{-1}\Vcal_B^0[\bJ]+\Ocal(\omega^2),\label{eq:highordE01}
  \end{align}
and
  \beq\label{eq:highordH01}
  \begin{split}
  \bH=&(I_3-D^2\Vcal_B^0\tilde\mu)^{-1}\nabla\times\Vcal_B^0[\bJ]+\omega(I_3-D^2\Vcal_B^0\tilde\mu)^{-1}\nabla\times\Vcal_B^0[\tilde\gamma\bE]\\
  &+k_0^2(I_3-D^2\Vcal_B^0\tilde\mu)^{-1}\Big(\nabla\times\Lcal_B^0+\Vcal_B^0\tilde\mu
(I_3-D^2\Vcal_B^0\tilde\mu)^{-1}\nabla\times\Vcal_B^0\Big)[\bJ]\\
  &+\Ocal(\omega^3).
  \end{split}
  \eeq
\end{cor}
\begin{proof}
By using \eqnref{eq:defnew}, \eqnref{eq:invA} and \eqnref{eq:solexp1} one can readily show that \eqnref{eq:fist01} and \eqnref{eq:fist02} hold in $B$. If $\nabla\cdot\bJ=0$, then $\bJ\in H_0(\mbox{div}; B)$ and by \eqnref{eq:opeig02}, one obtains $D^2\Vcal_B^0[\bJ]=0$ and $D^2\Lcal_B[\bJ]=0$. By using the higher order expansion in \eqnref{eq:solexp1}, one can show that \eqnref{eq:highordE01} and \eqnref{eq:highordH01} hold in $B$.

The proof is complete.
\end{proof}
\begin{rem}\label{re:rvs01}
We shall make use of the four asymptotic expansions contained in Corollary \ref{co:01} to recover the source $\bJ$ and the material parameters $\epsilon$ and $\mu$ as they are contained in the coefficients of the expansions. It is noted that those coefficients satisfy certain integral identities, and the unknowns are coupled together in those identities. We would also like to point out that one cannot have $\nabla\cdot \bJ=0$ and $\nabla\times \bJ=0$ hold simultaneously. Indeed, noting that $\bJ$ is compactly supported in $B$, and if both $\nabla\cdot \bJ=0$ and $\nabla\times \bJ=0$, then one must have that $\bJ\equiv 0$ by integrating by parts. Using this observation, together with the fact that $\nabla\times\Vcal_B^0[\bJ]=\Vcal_B^0[\nabla\times\bJ]$, we can then assume that the leading-order term in \eqref{eq:highordH01} is not vanishing.
%
%We mention that the leading order terms in \eqnref{eq:fist01} and \eqnref{eq:highordE01} are of $\Ocal(\omega)$ if $\Gs\neq 0$. In fact, by \eqnref{eq:fist01} and unique extension one easily has
%\beq\label{re:tmp010101}
%(I_3- D^2\Vcal_B^0\tilde\gamma)[\omega\bE]=i/\epsilon_0D^2\Vcal_B^0[\bJ]+\Ocal(\omega^2), \quad \mbox{in} \RR^3.
%\eeq
%Follows the proof of Lemma \ref{le:01} (\eqnref{eq:holap01} and \eqnref{eq:holap02}), one thus has $\omega\bE=\Ocal(\omega)$.

\end{rem}

\section{Unique recovery results for the inverse problem}

In this section, we present the main unique recovery results for the inverse problem \eqref{eq:ip2}, which are contained in Theorems~\ref{th:unique1} and \ref{th:unique2}. To that end, we first derive three critical auxiliary lemmas.

Henceforth, for a bounded domain $\Omega$ and a distribution $\psi$, we use $\psi|_{\partial\Omega}^+$ and $\psi|_{\partial\Omega}^-$, respectively, to signify the traces of $\psi$ on $\partial\Omega$ when one approaches $\partial\Omega$ from outside and inside of $\Omega$. The first auxiliary result is given as follows.
 \begin{lem}
Suppose that $\epsilon, \mu$ and $\sigma$ are similar to those described in Theorem~\ref{th:solexp1}, and particularly $\Gs=c\epsilon$ in $B$ with $c\geq 0$ a constant. Also, suppose that $\bJ\in H_0(\mathrm{div}; B)$, which satisfies \eqnref{eq:outbd1}. If $\nabla\cdot \bJ\neq 0$, then for $\omega\in\mathbb{R}_+$ sufficiently small, $\bE$ can be represented by $\bE=\nabla (u_1+u_2)+\Ocal(\omega)$ in $B$, where $u_1$ and $u_2$ satisfy
\beq\label{eq:solE01}
\begin{split}
u_1&=\left\{
\begin{array}{ll}
\ds \int_{\p B} N_{\epsilon}(\cdot,\By)(\nu\cdot \epsilon\bE\big|_{\p B}^{-})(\By)d s_\By & \mbox{in} \quad B\\
\ds -\Scal_B^0\GL_{1}^{-1} (\GL_{1}-\GL_{\epsilon})[\nu\cdot \epsilon\bE\big|_{\p B}^{-}]+C_1 & \mbox{in}\quad \RR^3\setminus\overline{B},
\end{array}
\right.  \\
u_2&=\left\{
\begin{array}{ll}
\ds c_1(\omega)\int_B N_{\epsilon}(\cdot,\By)(\nabla\cdot \bJ)(\By)d\By & \mbox{in} \quad B\\
\ds c_1(\omega)\Scal_B^0\GL_{1}^{-1}\Big[\int_B N_{\epsilon}(\cdot,\By)(\nabla\cdot \bJ)(\By)d\By\Big] +C_2 & \mbox{in}\quad \RR^3\setminus\overline{B},
\end{array}
\right.
\end{split}
\eeq
with $C_1$ and $C_2$ two constants, respectively, given by
$$
C_1=\frac{1}{|\p B|}\int_{\p B}\Scal_B^0\GL_{1}^{-1} (\GL_{1}-\GL_{\epsilon})[\nu\cdot \epsilon\bE\big|_{\p B}^{-}](\By)d s_\By,
$$
and
$$
C_2=-\frac{c_1(\omega)}{|\p B|}\int_{\p B}\ \Scal_B^0\GL_{1}^{-1}\Big[\int_B N_{\epsilon}(\cdot,\By)(\nabla\cdot \bJ)(\By)d\By\Big](\Bx)d s_\Bx,
$$
and $c_1(\omega)=i\epsilon_0^{-1}(\omega+ci)^{-1}$. If $\nabla\cdot \bJ=0$ then $\bH$ can be represented as
\beq\label{eq:solH01}
(I_3+D^2\Vcal_B^0)[\bH]=-\mu_0^{-1}\nabla\Scal_B^0[(\nu\cdot \mu\bH)|_{\p B}^{-}]+\Vcal_B^0[\nabla\times \bJ]+\Ocal(\omega^2) \quad \mbox{in} \quad B.
\eeq
\end{lem}
\begin{proof}
Recall from \eqnref{eq:fist01} that $\bE\in H(\mbox{curl}; B)$ and satisfies
\begin{equation}\label{eq:lll2}
(I_3- D^2\Vcal_B^0\tilde\gamma)[\omega\bE]=i\epsilon_0^{-1}D^2\Vcal_B^0[\bJ]+\Ocal(\omega^2) \quad \mbox{in} \quad B.
\end{equation}
By taking $\mathrm{curl}$ of both sides of \eqref{eq:lll2}, one readily has that
\beq\label{eq:transf02}
\nabla\times \bE=\Ocal(\omega).
\eeq
Noting $\bJ\in H_0(\mbox{div}; B)$, one also has from taking $\mathrm{div}$ of both sides of \eqref{eq:lll2} that
\beq\label{eq:transf01}
\nabla\cdot (1+\tilde\gamma)\bE=-i\omega^{-1}\epsilon_0^{-1}\nabla\cdot \bJ+\Ocal(\omega).
\eeq
Let $u$ be the solution to
\beq\label{eq:lowfre1}
\left\{
\begin{array}{ll}
\nabla\cdot(\epsilon\nabla u) =-c_1(\omega)\nabla\cdot\bJ  & \mbox{in} \quad B\medskip \\
\Delta u =0 & \mbox{in } \RR^3 \setminus \overline{B} \\
\ds \frac{\p u}{\p \nu} \Big|_+=\epsilon \frac{\p u}{\p \nu} \Big|_- & \mbox{on }  \p B\\
u|_+=u|_- & \mbox{on } \partial B\\
u(x)=O(\|x\|^{-1}) & \mbox{as } \|x\| \rightarrow \infty
\end{array}
\right. ,
\eeq
where $c_1(\omega)=i\epsilon_0^{-1}(\omega+ci)^{-1}$. Noting that $\Gs=c\epsilon$ in $B$, and by combining \eqnref{eq:transf02}, \eqnref{eq:transf01} and \eqnref{eq:lowfre1}, we thus have $\bE=\nabla u+\Ocal(\omega)$.
The solution $u$ to \eqnref{eq:lowfre1} can be written as $u=u_1+u_2$ (see, e.g. \cite{ADKL14}), where
\beq\label{eq:sol01}
u_1(\Bx)=\left\{
\begin{array}{ll}
\ds \int_{\p B} N_{\epsilon}(\Bx,\By)(\nu\cdot\epsilon \bE|_{\p B}^{-})(\By)d s_\By, & \Bx\in B\\
\ds -\Scal_B^0\GL_{1}^{-1} (\GL_{1}-\GL_{\epsilon})[\nu\cdot\epsilon \bE|_{\p B}^{-}](\Bx)+C_1, & \Bx\in \RR^3\setminus\overline{B}
\end{array}
\right. ,
\eeq
and
\beq\label{eq:sol02}
u_2=\left\{
\begin{array}{ll}
\ds c_1(\omega)\int_B N_{\epsilon}(\Bx,\By)(\nabla\cdot \bJ)(\By)d\By, & \Bx\in B\\
\ds c_1(\omega)\Scal_B^0\GL_{1}^{-1}\Big[\int_B N_{\epsilon}(\cdot,\By)(\nabla\cdot \bJ)(\By)d\By\Big](\Bx)+C_2, & \Bx\in \RR^3\setminus\overline{B}
\end{array}
\right. ,
\eeq
with $C_1$  and $C_2$ two constants, respectively, given by
$$
C_1=\frac{1}{|\p B|}\int_{\p B}\Scal_B^0\GL_{\epsilon_0}^{-1} (\GL_{\epsilon_0}-\GL_{\epsilon})[\nu\cdot \epsilon\bE\big|_{\p B}^{-1}](\By)d s_\By,
$$
and
$$
C_2=-\frac{c_1(\omega)}{|\p B|}\int_{\p B}\Scal_B^0\GL_{\epsilon_0}^{-1}\Big[\int_B N_{\epsilon}(\cdot,\By)(\nabla\cdot \bJ)(\By)d\By\Big](\Bx)ds_\Bx.
$$
Therefore, \eqnref{eq:solE01} is proved.

Next, if $\nabla\cdot \bJ=0$, then by \eqnref{eq:highordH01} and noting that
$$
\nabla\times\Vcal_B^0[\tilde\gamma\bE]=\Vcal_B^0[\tilde\gamma\nabla\times\bE]=\Ocal(\omega^2),
$$
we have
\beq\label{eq:tmp04}
(I_3+ D^2\Vcal_B^0)[\bH]= D^2\Vcal_B^0[(1+\tilde\mu)\bH]+\nabla\times\Vcal_B^0[\bJ]+\Ocal(\omega^2).
\eeq
It can be readily verified that
  \beq\label{eq:tmp05}
  \nabla\cdot (1+\tilde\mu)\bH=\Ocal(\omega^2).
  \eeq
Noting that $1+\tilde\mu=\mu/\mu_0$, by \eqnref{eq:tmp04}, \eqnref{eq:tmp05} and integration by parts, we obtain
\begin{align*}
&(I_3+ D^2\Vcal_B^0)[\bH]\\
=&-\mu_0^{-1}\nabla\int_B\nabla_y\Gamma_0(\cdot-\By)\cdot (\mu\bH)(\By)d s_\By +\Vcal_B^0[\nabla\times \bJ]+\Ocal(\omega^2) \\
= &-\mu_0^{-1}\nabla\int_B\nabla_y\cdot(\Gamma_0(\cdot-\By) (\mu\bH)(\By))d s_\By +\Vcal_B^0[\nabla\times \bJ]+\Ocal(\omega^2)\\
= &- \mu_0^{-1}\nabla\Scal_B^0[(\nu\cdot\mu\bH)|_{\p B}^{-}]+\Vcal_B^0[\nabla\times \bJ]+\Ocal(\omega^2).
\end{align*}

The proof is complete.
\end{proof}

Next we present the second auxiliary result.

\begin{lem}\label{le:au01}
Let $N_\varepsilon(\Bx,\By)$ be defined in \eqnref{revn}, and suppose that $\varepsilon$ is a given positive constant in $B$. Then there holds the following identity,
\beq\label{eq:au01}
\begin{split}
&\Big(-\frac{I}{2}+\Kcal_B^0\Big)\Big[\int_{B}N_{\varepsilon}(\cdot-\By)(\nabla\cdot \bJ)(\By)d\By\Big](\Bx)\\
=& \varepsilon^{-1}\int_{B}\Gamma_0(\Bx-\By)(\nabla\cdot \bJ)(\By)d\By, \quad\Bx\in \p B.
\end{split}
\eeq
\end{lem}
\begin{proof}
For any $g\in L^2(\p B)$, we let $u$ be defined by
$$
u(\By):=\int_{\p B}\Big(-\frac{I}{2} +\Kcal_B^0\Big)[N_{\varepsilon}(\cdot,\By)](\Bx)g(\Bx)d s_\Bx, \quad \By\in B.
$$
Then we conclude
$$
u(\By):=\int_{\p B}N_{\varepsilon}(\Bx,\By)\Big(-\frac{I}{2} +(\Kcal_B^0)^*\Big)[g](\Bx)d s_\Bx, \quad \By\in B.
$$
Noticing that $\varepsilon$ is a constant, we have $\Delta u=0$ and
$$
\nu\cdot\nabla u=\varepsilon^{-1} \Big(-\frac{I}{2} +(\Kcal_B^0)^*\Big)[g], \quad \int_{\p B} u=0, \quad \mbox{on} \quad\p B.
$$
{It is easily seen that $u$ can also be represented by the formula
$$u=\varepsilon^{-1}\Scal_B^{0}[g]-\varepsilon^{-1}\int_{\p B}\Scal_B^{0}[g](\Bx)d s_\Bx.$$
By using that $\int_B\nabla\cdot\bJ=0$ we thus have
\[
\begin{split}
&\int_B\int_{\p B}\Big(-\frac{I}{2} +\Kcal_B^0\Big)[N_{\varepsilon}(\cdot,\By)](\Bx)g(\Bx)d s_\Bx(\nabla\cdot\bJ)(\By)d\By\\
=&\varepsilon^{-1}\int_B\Scal_B^0[g](\By)(\nabla\cdot\bJ)(\By)d\By.
\end{split}
\]}
Finally, by interchanging the order of integrations and using the fact that $g\in L^2(\p B)$ is arbitrary, we readily have \eqnref{eq:au01}.

The proof is complete.
\end{proof}
\begin{rem}
It is remarked that Lemma \ref{le:au01} is adapted and modified from Lemma 2.28 in \cite{HK07:book}. Here for the sake of convenience to the readers and self-containedness of the paper, we present its proof.
\end{rem}

Next we introduce two definitions.

\begin{defn}\label{df:01}
We call a harmonic function $u$ in $B$ a {\it Herglotz harmonic function} with $\xi\in \mathbb{C}^3$ if
 \beq\label{herglotzfn01}
  u(\Bx)=\alpha e^{i\Bx\cdot \xi} +\beta, \quad \Bx\in B,
  \eeq
where $\xi\neq \mathbf{0}$, $\xi\cdot\xi=0$ and $\alpha, \beta \in\mathbb{C}$ are two constants.
\end{defn}

\begin{defn}\label{df:02}
We call a function $\Phi\in H^1(B)$ \emph{admissible}, if one of the following two conditions is fulfilled:
      \begin{enumerate}
        \item[(i)] $\Phi(x)=h(x)$ for $x\in B$, and $h$ is a harmonic function in $\mathbb{R}^3$;
        \item[(ii)] There exists a unit vector $\mathbf{d}\in\mathbb{S}^2$ such that $\mathbf{d}\cdot\nabla\Phi=0$.
      \end{enumerate}
\end{defn}
%We mention that the second case (ii) in Definition \ref{df:02} imply that $\nabla \Phi$ is a rotational invariant vectorial function. In fact, suppose $\mathbf{d}$ satisfies (ii), then by vector calculus one has
%$$
%\mathbf{d}\times \nabla \Phi\times \mathbf{d}=\nabla\Phi-(\mathbf{d}\cdot\nabla\Phi) \mathbf{d}= \nabla\Phi,
%$$
%which shows the rotational invariance of $\nabla \Phi$.
The following lemma shall also be needed.
\begin{lem}\label{le:intrans01}
Let $(\bE, \bH)$ be the solution to \eqnref{eq:pss} and \eqnref{eq:radia2}. Suppose that ($\epsilon$, $\mu$, $\Gs$) and $\bJ\in H_0(\mathrm{div}; B)$ satisfy \eqnref{eq:outbd1} and $\nabla\cdot \bJ=0$. Suppose further that $\epsilon$ and $\Gs$ are piecewise constants in $B$ which satisfy \eqnref{eq:multilayered} in Assumption \ref{asm:01}. Then there holds the following transmission condition,
\beq\label{eq:addinnertran1}
\nu\cdot (\epsilon^{(j)}+i\sigma^{(j)}/\omega)\bE |_+= \nu\cdot (\epsilon^{(j+1)}+i\sigma^{(j+1)}/\omega)\bE |_- \quad \mbox{on} \quad \p \Sigma_{j},
\eeq
for $j=0, 1, \ldots, N$. {Moreover, there also holds
\beq\label{eq:addinnertran01}
\nu\cdot \mu_0\bH|_+=\nu\cdot \mu\bH|_- \quad \mbox{on} \quad \p B.
\eeq
}
\end{lem}
\begin{proof}
We only prove \eqnref{eq:addinnertran1}, and \eqnref{eq:addinnertran01} can be proved by using similar arguments. By taking the divergence of both sides of the second equation in \eqnref{eq:pss} and noting that $\nabla\cdot \bJ=0$ by the assumption, we have
\beq\label{eq:addinnertran2}
\nabla\cdot (\epsilon+i\sigma/\omega)\bE=\nabla\cdot\bJ=0 \quad \mbox{in} \quad \RR^3.
\eeq
By taking the inner product of both sides of the second equation in \eqnref{eq:pss} with the gradient of a test function $\psi\in C_0^{\infty}(\RR^3)$, and integrating both sides over $\RR^3$, there holds
\beq\label{eq:addinnertran3}
\int_{\RR^3} (\nabla\times \bH) \cdot \nabla\psi=-i\omega\int_{\RR^3} (\epsilon+i\sigma/\omega)\bE\cdot \nabla\psi+ \int_{\RR^3}\bJ\cdot \nabla\psi.
\eeq
By using the vector calculus identity and Green's formula, the LHS of \eqnref{eq:addinnertran3} can be rewritten
\beq\label{eq:addinnertran4}
\int_{\RR^3} (\nabla\times \bH) \cdot \nabla\psi=\int_{\RR^3} \nabla\cdot (\bH\times\nabla\psi)=0.
\eeq
Using \eqnref{eq:addinnertran2} and Green's formula, the RHS of \eqnref{eq:addinnertran3} implies
\beq\label{eq:addinnertran5}
\begin{split}
& -i\omega\int_{\RR^3} (\epsilon+i\sigma/\omega)\bE\cdot \nabla\psi+ \int_{\RR^3}\bJ\cdot \nabla\psi \\
=&-i\omega\int_{\RR^3\setminus\overline{\Sigma}_j}(\epsilon+i\sigma/\omega)\bE\cdot \nabla\psi-i\omega\int_{\Sigma_j} (\epsilon+i\sigma/\omega)\bE\cdot \nabla\psi\\
=&i\omega\int_{\p \Sigma_j}\nu\cdot(\epsilon+i\sigma/\omega)\bE\Big|_+\psi
-i\omega\int_{\p \Sigma_j}\nu\cdot(\epsilon+i\sigma/\omega)\bE\Big|_-\psi.
\end{split}
\eeq
Substituting \eqnref{eq:addinnertran4} and \eqnref{eq:addinnertran5} into \eqnref{eq:addinnertran3}, we have
\beq\label{eq:addinnertran6}
\int_{\p \Sigma_j}\nu\cdot(\epsilon+i\sigma/\omega)\bE\Big|_+\psi=\int_{\p \Sigma_j}\nu\cdot(\epsilon+i\sigma/\omega)\bE\Big|_-\psi.
\eeq
Finally, since $\psi$ is arbitrary, we thus have \eqnref{eq:addinnertran1}, which completes the proof.
\end{proof}

We are ready to present the first main theorem that contains the results on the simultaneous recovery of $\epsilon$, $\mu$, $\sigma$ and $\bJ$ from the knowledge of $\Pi_{\epsilon,\mu, \sigma, \bJ}$.

\begin{thm}\label{th:unique1}
Let $(\epsilon_1, \mu_1, \Gs_1,\bJ_1)$ and $(\epsilon_2, \mu_2, \Gs_2, \bJ_2)$ be two sets of EM configurations. Suppose $\epsilon_j>0$, $\mu_j>0$, $\bJ_j\in H_0(\mathrm{div}; B)$ which verify \eqnref{eq:outbd1}. Suppose further that $\epsilon_j, \mu_j$ and $\sigma_j$ are constants in $B$. Let $(\bE_1, \bH_1)$ and $(\bE_2, \bH_2)$ be the corresponding solutions to \eqnref{eq:pss} and \eqnref{eq:radia2}, associated with $(\epsilon_1, \mu_1, \Gs_1, \bJ_1)$ and $(\epsilon_2, \mu_2, \Gs_2, \bJ_2)$, respectively.
If
\begin{equation}\label{eq:lll3}
\Pi_{\epsilon_1,\mu_1,\Gs_1,\bJ_1}(\Bx, \omega)=\Pi_{\epsilon_2, \mu_2, \Gs_2, \bJ_2}(\Bx, \omega), \quad (\Bx, \omega)\in \partial B\times (0, \omega_0),
\end{equation}
where $\omega_0$ is any given positive constant, then we have the following results:
\begin{enumerate}
  \item[(i)]
{ Set $\Phi:=\nabla\cdot (\bJ_1-\bJ_2)$. If $\nabla \cdot \bJ_j\neq 0$ and
{ \beq\label{eq:nozerocond1}
 \nu\cdot\bE_j|_{\p B}\neq 0, \quad \nu\cdot\bH_j|_{\p B}^++\nu\cdot\bH_j|_{\p B}^{+}\omega^{-1}\neq 0\ \ \mbox{as} \quad \omega\rightarrow +0,
 \eeq}
$j=1, 2$, and $\gamma_1=\gamma_2$, then $\Phi=0$  if $\Phi$ satisfies the {admissibility} condition in Definition \ref{df:02}.
    Furthermore, if $\nabla\times\bJ_1=\nabla\times\bJ_2=0$, then $\bJ_1=\bJ_2$ and $\mu_1=\mu_2$.}\\
  \item[(ii)] If $\nabla \cdot \bJ_j= 0$, and $\Scal_B^0[\nu\cdot\bH_j|_{\p B}^{-}]$ are Herglotz harmonic functions with $\xi_j$, $j=1, 2$ and
{       \beq\label{eq:nozerocond2}
       \nu\cdot\bE_j|_{\p B}^+\omega^{-1}\neq 0\ \ \mbox{as} \quad \omega\rightarrow +0,
       \eeq }
       $j=1, 2$, then we have that $\bJ_1=\bJ_2$, $\mu_1=\mu_2$, $\Gs_1=\Gs_2$ and $\epsilon_1=\epsilon_2$ provided $\Phi:=\xi\cdot(\nabla\times (\bJ_1-\bJ_2))$, where $\xi\cdot\xi_1=\xi\cdot\xi_2=0$, satisfies the {admissibility} condition in Definition \ref{df:02}.
\end{enumerate}
\end{thm}
\begin{proof}
We first prove case~(i). Since $\epsilon_j$, $j=1, 2$ are constants in $B$, and $\nabla \cdot \bJ_j\neq 0$, $j=1, 2$, by using \eqnref{eq:lll3}, or $\nu\cdot\epsilon_1\bE_1=\nu\cdot\epsilon_2\bE_2$ (see \eqnref{eq:addinnertran1} and note that $\gamma_1=\gamma_2$) on $\p B$, the first line in \eqnref{eq:sol01}, \eqnref{eq:sol02}, \eqnref{revn} and $\epsilon_1=\epsilon_2$, we have
  $$
  \int_B N_{\epsilon_1}(\Bx,\By)(\nabla\cdot (\bJ_1-\bJ_2))(\By)d\By=0, \quad \Bx\in \p B.
  $$
  Using \eqnref{eq:au01} we obtain
  \beq\label{eq:uniqueres01}
  \int_B \Gamma_0(\Bx,\By)(\nabla\cdot (\bJ_1-\bJ_2))(\By)d\By=0, \quad \Bx\in \p B.
  \eeq
  Note that $\Gamma_0(\Bx,\By)$ is a harmonic function which decays at infinity (with order $\|\Bx\|^{-1}$) for $\By\in B$ and $\Bx\in \RR^3\setminus B$. We easily see 
  \beq\label{eq:uniqueres001}
  \int_B \Gamma_0(\Bx,\By)(\nabla\cdot (\bJ_1-\bJ_2))(\By)d\By=0, \quad \Bx\in \RR^3\setminus B.
  \eeq
  Recall the following addition formula for $\|\Bx\|>\|\By\|$ (see \cite{CK,Ned})
  \beq\label{eq:uniqueres02}
  \frac{1}{4\pi\|\Bx-\By\|}=\sum_{n=0}^\infty \sum_{m=-n}^{n} \frac{1}{2n+1}Y_n^m(\Bx/\|\Bx\|) \overline{Y_n^m(\By/\|\By\|)} \,\frac{\|\By\|^n}{\|\Bx\|^{n+1}},
  \eeq
  where $Y_n^m$ denotes the spherical harmonic function of degree $n$ and order $m$. By choosing $\Bx\in \p\mathcal{B}_R$, where $\mathcal{B}_R$ is a sufficiently large ball such that $B\Subset \mathcal{B}_R$, and inserting \eqnref{eq:uniqueres02} into \eqnref{eq:uniqueres001} and using the orthogonality of $Y_n^m$, we immediately conclude that
  \beq\label{eq:uniqueres03}
  \int_B \|\By\|^nY_n^m(\By/\|\By\|)(\nabla\cdot (\bJ_1-\bJ_2))(\By)d\By=0, \quad \Bx\in \p \mathcal{B}_R.
  \eeq
  Note that $\|\By\|^nY_n^m(\By/\|\By\|)$, $m=-n,\cdots,n$, $n=0, 1, 2, \cdots$, yield all the homogeneous harmonic polynomials. Hence \eqnref{eq:uniqueres03} and the first case (i) in Definition \ref{df:02} imply $\nabla
  \cdot (\bJ_1-\bJ_2)=0$; see also \cite{LiuUhl15} for the relevant argument for a similar case in thermo- and photo-acoustic tomography. For the second case (ii) in Definition \ref{df:02}, since $\epsilon$, $\mu$ and $\sigma$ are constants in $B$, due to the rotational invariance of the Maxwell system (see \cite{BLZ}), without loss of generality, we can assume that $\mathbf{d}=(0, 0, 1)$ in the case (ii), which gives
  $$
  \p_{\Bx_3}\Phi(\Bx)=0\quad\mbox{where}\ \ \Bx=(\Bx_j)_{j=1}^3\in\mathbb{R}^3.
  $$
That is, $\Phi$ is independent of $\Bx_3$. Following a similar argument to the proof of Theorem 2.2 in \cite{LiuUhl15} by using Fourier analysis techniques, we can show that \eqnref{eq:uniqueres03} and the second case (ii) in Definition \ref{df:02} also imply that $\nabla\cdot (\bJ_1-\bJ_2)=0$.
 By using \eqnref{eq:solE01} one can easily see that $\bE_1-\bE_2=\Ocal(\omega)$ in $B$.
% Together with
%  $$\nu\cdot(1+\tilde\gamma_1)\bE_1|_{\p B}^{-}=\nu\cdot(1+\tilde\gamma_2)\bE_2|_{\p B}^{-},$$
%  which is equivalent to
%  \beq\label{eq:revisereconep02}
%  \nu\cdot\Gs_1\bE_1|_{\p B}^{-}-\nu\cdot\Gs_2\bE_2|_{\p B}^{-}=-i\omega(\nu\cdot\epsilon_1\bE_1|_{\p B}^{-}-\nu\cdot\epsilon_2\bE_2|_{\p B}^{-}),
%  \eeq
%  and the condition $\epsilon_1=\epsilon_2$, one has
%  $$
% (\Gs_1-\Gs_2)\nu\cdot\bE_1|_{\p B}^{-}=\Ocal(\omega).
%  $$
%  Since $\nu\cdot \bE_1\neq 0$ as $\omega\rightarrow +0$ on $\p B$ one thus has $\Gs_1=\Gs_2$.

  If $\bJ_j\in H(\mbox{div}; B)$, $j=1,2$, are curl-free functions, then by using the Helmholtz decomposition we have
  $$
  \bJ_j= \nabla{u_j}, \quad j=1, 2,
  $$
  which together with $\nabla\cdot (\bJ_1-\bJ_2)=0$ and $\bJ_1=\bJ_2=0$ on $\p B$, readily implies that $\bJ_1=\bJ_2$ in $B$.
  Next, by using integration by parts, we can obtain
  $$\nabla\times\Vcal_B^0[\bJ_j]=\Vcal_B^0[\nabla\times\bJ_j]=0, \quad j=1, 2,$$
  which shows that the second term in \eqnref{eq:fist02} vanishes.
   By \eqnref{eq:fist02} we thus have
$$
\bH_j=(I_3-D^2\Vcal_B^0\tilde\mu_j)^{-1}\nabla\times\Vcal_B^0[-i(\epsilon_j-\epsilon_0)\omega\bE_j+\Gs_j\bE_j+\bJ_j]+\Ocal(\omega^2), \quad j=1, 2,
$$
and hence
\beq\label{eq:transf03}
(I_3-D^2\Vcal_B^0\tilde\mu_j)[\bH_j]=\nabla\times\Vcal_B^0[-i(\epsilon_j-\epsilon_0)\omega\bE_j+\Gs_j\bE_j+\bJ_j]+\Ocal(\omega^2), \quad j=1,2 .
\eeq
  Noting that $\mu_j$, $j=1, 2$, are constants in $B$, and by taking divergence of both sides of \eqnref{eq:transf03}, one can show that
  there holds
  $$
   \nabla\cdot \bH_j=\Ocal(\omega^2), \quad j=1, 2.$$
By \eqnref{eq:lll3}, or $\nu\cdot(\mu_1\bH_1-\mu_2\bH_2)=0$ on $\p B$ (cf. \eqnref{eq:addinnertran01}) and \eqnref{eq:transf03}, we have from integration by parts that
  \beq\label{eq:forref01}
  (\bH_1-\bH_2)|_{\p B}^{-}-\nabla\Scal_B^0[\nu\cdot(\bH_1-\bH_2)|_{\p B}^{-}]=\Ocal(\omega^2).
  \eeq
 Using the trace formula \eqnref{eq:trace} one thus has
  $$
  (\mu_1^{-1}-\mu_2^{-1})\left(\frac{1}{2}I-(\Kcal_B^0)^*\right)[\nu\cdot\bH_1|_{\p B}^{+}]=\Ocal(\omega^2).
  $$
  Since $1/2 I-(\Kcal_B^0)^*$ is invertible on $L^2(\p B)$ we finally obtain $\mu_1=\mu_2$.\\

 For case (ii), if $\nabla\cdot \bJ_j=0$, $j=1, 2$, then by using \eqnref{eq:tmp05} and noting that $\mu_j$, $j=1, 2$, are constants one has $\nabla\cdot \bH_j=\Ocal(\omega^2)$, and then by using \eqnref{eq:solH01} and integration by parts one further obtains
   \beq\label{eq:consttmp01}
   \bH_j=-\tilde\mu_j\nabla\Scal_B^0[(\nu\cdot\bH_j)|_{\p B}^{-}]+\Vcal_B^0[\nabla\times \bJ_j]+\nabla\times\Vcal_B^0[\Gs_j\bE_j]+\Ocal(\omega^2) \quad \mbox{in} \ B,
   \eeq
   where $j=1, 2$.
 Hence by using \eqnref{eq:lll3} and \eqnref{herglotzfn01} one obtains
  $$-i(\tilde\mu_1\xi_1 e^{i\Bx\cdot \xi_1}-\tilde\mu_2\xi_2 e^{i\Bx\cdot \xi_2})+ \Vcal_B^0[\nabla\times(\bJ_1-\bJ_2)]=0 \quad  \mbox{on} \ \p B.$$
  Taking the inner product with $\xi$ on both sides of the above equation (noting that $\xi\cdot\xi_j=0$, $j=1, 2$), one then has
  $$\Vcal_B^0[\xi\cdot(\nabla\times(\bJ_1-\bJ_2))]=0 \quad  \mbox{on} \ \p B.$$
 Thus by using the fact that $\xi\cdot(\nabla\times (\bJ_1-\bJ_2))$ satisfies the {admissibility condition} in Definition \ref{df:02} and similar analysis as that for \eqnref{eq:uniqueres01}, one can show that
  $$
  \xi\cdot(\nabla\times (\bJ_1-\bJ_2))=0 \quad \mbox{in} \ B.
  $$
  Therefore $\nabla\times(\bJ_1-\bJ_2)$ can be denoted by $\xi'g$, where $g\in L^2(B)^3$ and $\xi'\in \mathbb{C}^3$ satisfies
  $$
  \left(
  \begin{array}{l}
  \Re{\xi'}\\
  \Im{\xi'}
  \end{array}
  \right)\in \mathbb{M}(\xi),
  $$
  with $\mathbb{M}(\xi)\subset\RR^6$ given by
  $$\mathbb{M}(\xi):=\mbox{span}\left\{
  \left(
  \begin{array}{l}
  \Re{\xi}\\
  \Im{\xi}
  \end{array}
  \right),
  \left(
  \begin{array}{l}
  \Im{\xi}\\
  -\Re{\xi}
  \end{array}
  \right),
  \left(
  \begin{array}{c}
  \Re{\xi}\times\Im{\xi}\\
  0
  \end{array}
  \right),
    \left(
  \begin{array}{c}
  0\\
  \Re{\xi}\times\Im{\xi}
  \end{array}
  \right)
  \right\}.
  $$
  Moreover, $\xi'g$ satisfies $\nabla g\cdot\xi'=0$ and thus $g(\Bx)=\xi''\cdot\Bx$, $\xi''\in \mathbb{M}(\xi')$. To sum up, we have
  $$
  \nabla\times(\bJ_1-\bJ_2)=\xi'(\xi''\cdot \Bx), \quad \xi'\in \mathbb{M}(\xi), \quad \xi''\in \mathbb{M}(\xi').
  $$
  Noting the fact that $\bJ_j=0$ on $\p B$, $j=1, 2$, one immediately has $\xi'(\xi''\cdot \Bx)=0$ and hence $\nabla\times(\bJ_1-\bJ_2)=0$, which together with the assumption that $\nabla\cdot \bJ_j=0$, $j=1, 2$, readily implies $\bJ_1=\bJ_2$. By using \eqnref{eq:consttmp01} again, one can readily have that $\mu_1=\mu_2$ and thus $\bH_1-\bH_2=\Ocal(\omega)$ in $B$.
  Define $\bF_j$ in $B$ by
  \beq\label{eq:newadd1}
  \bF_j:=\nabla\times\Vcal_B^0[\tilde\mu_j\bH_j]+\Vcal_B^0[\bJ_j], \quad j=1, 2,
  \eeq
  then \eqnref{eq:highordE01} can be reformulated as
  \beq\label{eq:newadd2}
  (I_3-D^2\Vcal_B^0\tilde\gamma_j)[\bE_j]=i\omega\mu_0\bF_j+\Ocal(\omega^2), \quad j=1, 2.
  \eeq
  By integration by parts we then have that (noting that $\tilde\gamma_j$, $j=1, 2$ are constants in $B$)
  $$
  \bE_j+\nabla\Scal_B^0[\nu\cdot\tilde\gamma_j\bE_j|_{\p B}^{-}]= i\omega\mu_0\bF_j + \Ocal(\omega^2), \quad j=1, 2
  $$
  holds in $B$. Since we have already proved $\bF_1-\bF_2=\Ocal(\omega)$, by taking the trace of the above equation and using \eqnref{eq:lll3}, or $\nu\cdot(1+\tilde\gamma_1)\bE_1=\nu\cdot(1+\tilde\gamma_2)\bE_2$ on $\p B$ (cf. \eqnref{eq:addinnertran1}), we have
  \beq\label{eq:forref02}
  \left(\frac{I}{2}-(\Kcal_B^0)^*\right)[\nu\cdot\bE_1|_{\p B}^{-}-\nu\cdot\bE_2|_{\p B}^{-}]=\Ocal(\omega^2),
  \eeq
%  and
%  \beq\label{eq:forref03}
%  i\omega^{-1}\epsilon_0^{-1}(\Gs_1\epsilon_1^{-1}-\Gs_2\epsilon_2^{-1})\left(\frac{I}{2}+(\Kcal_B^0)^*\right)[\nu\cdot\bE_1|_{\p B}^{+}]=\Ocal(\omega^2).
%  \eeq
By virtue of the invertibility of $I/2-(\Kcal_B^0)^*$ on $L^2(\p B)$ we immediately have from \eqnref{eq:forref02} that
  $$\nu\cdot\bE_1|_{\p B}^{-}-\nu\cdot\bE_2|_{\p B}^{-}=\Ocal(\omega^2).$$
Then using \eqnref{eq:addinnertran1} on $\p B$ and the assumption that $\nu\cdot \bE_j/\omega\neq 0$ as $\omega\rightarrow +0$, we have
  $\Gs_1=\Gs_2$.
  By using \eqnref{eq:consttmp01} again, one further has
  \beq\label{eq:revisereconep03}
  \bH_1-\bH_2=\Ocal(\omega^2),
  \eeq
  and thus there holds
  \beq\label{eq:revisereconep04}
  \nabla\times(\bE_1-\bE_2)=\Ocal(\omega^3).
  \eeq
  Since $\nabla\cdot\bJ=0$, we have $\nabla\cdot(1+\tilde\gamma_j)\bE_j=0$ (note that $1+\tilde\gamma_j$ is constant in $B$) and hence $\nabla\cdot(\bE_1-\bE_2)=0$. Then using the transmission condition $\nu\times\bE_j|_+=\nu\times\bE_j|_-$, $j=1, 2$, which actually implies $\nu\times(\bE_1-\bE_2)|_-=0$ on $\p B$, there exists $u\in H^1(B)$ such that $\bE_1-\bE_2=\nabla u+\Ocal(\omega^3)$ and
  $$\Delta u=0, \quad \nu\times\nabla u=0 \quad \mbox{on} \quad \p B.$$
 Hence there holds $u=C(\omega)$ and thus $\bE_1-\bE_2=\Ocal(\omega^3)$.
 Finally by using \eqnref{eq:addinnertran1} again one has $\epsilon_1=\epsilon_2$.

The proof is complete.
\end{proof}

Next we present the second main theorem on the simultaneous recovery when the medium parameters are piecewise constants (cf. Assumption~\ref{asm:01}). We first derive the following lemma.

\begin{lem}\label{le:lead01}
Suppose $\mu\neq \mu_0$ is a constant in $B$ and $\Gs$, $\epsilon$ are piecewise constants in $B$ which satisfy \eqnref{eq:multilayered} in Assumption \ref{asm:01} with $\Gs\neq 0$. Suppose further that $\nabla\cdot\bJ=0$. Let $\bE$ be defined in \eqnref{eq:highordE01}, and
\beq\label{eq:leadingodH01}
\bH^{(0)}:=(I_3-D^2\Vcal_B^0\tilde\mu)^{-1}\nabla\times\Vcal_B^0[\bJ].
\eeq
Then
\beq\label{eq:leadd0101}
\bE=(\nabla u+\nabla\times\Psi)\omega+\Ocal(\omega^2) \quad \mbox{in} \quad B,
\eeq
where $\Psi$ satisfies $\nabla\times\nabla\times\Psi=i\mu\bH^{(0)}$ and $u$ is the solution to
\beq\label{eq:leadd010101}
\left\{
\begin{array}{ll}
\nabla\cdot\Gs u=0 &\mbox{in} \quad B, \\
\nu\cdot\Gs\nabla u= -i\epsilon_0\nu\cdot\bE|_+-\nu\cdot (\Gs\nabla \times \Psi) &\mbox{on} \quad \p B.
\end{array}
\right.
\eeq
\end{lem}
\begin{proof}
Define
$$\Phi:=\nabla\times\Vcal_B^0\tilde\mu(I_3-D^2\Vcal_B^0\tilde\mu)^{-1}\nabla\times\Vcal_B^0[\bJ]\quad \mbox{in}\quad B,$$
then one has
\beq\label{eq:leadd0102}
\begin{split}
\nabla\times\Phi=&-\Delta\Vcal_B^0\tilde\mu(I_3-D^2\Vcal_B^0\tilde\mu)^{-1}\nabla\times\Vcal_B^0[\bJ]\\
&+ D^2\Vcal_B^0\tilde\mu(I_3-D^2\Vcal_B^0\tilde\mu)^{-1}\nabla\times\Vcal_B^0[\bJ]\\
=&\mu\mu_0^{-1}(I_3-D^2\Vcal_B^0\tilde\mu)^{-1}\nabla\times\Vcal_B^0[\bJ]-\nabla\times\Vcal_B^0[\bJ].
\end{split}
\eeq
On the other hand, one has from \eqnref{eq:highordE01} that
\beq\label{eq:leadd0103}
(I_3-D^2\Vcal_B^0\tilde\gamma)[\bE]=i\omega\mu_0(\Phi+\Vcal_B^0[\bJ])+\Ocal(\omega^2) \quad \mbox{in} \quad B.
\eeq
\eqnref{eq:leadd0102} and \eqref{eq:leadd0103} readily imply that
\beq\label{eq:leadd0104}
\nabla\times\bE=i\omega\mu(I_3-D^2\Vcal_B^0\tilde\mu)^{-1}\nabla\times\Vcal_B^0[\bJ]+\Ocal(\omega^2) \quad \mbox{in} \quad B,
\eeq
which in combination with the fact $\nabla\cdot\bJ=0$ further implies that
\beq\label{eq:leadd0105}
\nabla\cdot\gamma\bE=\Ocal(\omega^2).
\eeq
By \eqnref{eq:leadd0104} and \eqnref{eq:leadd0105}, one has
\begin{equation}\label{eq:ff1}
\nabla\times\bE=\Ocal(\omega), \quad \nabla\cdot\Gs\bE=\Ocal(\omega) \quad \mbox{in} \quad \RR^3.
\end{equation}
Then using \eqref{eq:ff1} and the transmission condition \eqnref{eq:addinnertran1} on $\p B$, namely
$$
i\nu\cdot\Gs\bE|_-=\omega\nu\cdot\bE|_+-\omega\nu\cdot\bE|_-=\Ocal(\omega) \quad \mbox{on} \quad \p B,
$$
one can easily show that $\bE=\Ocal(\omega)$ in $B$. Define
\beq\label{eq:leadd0106}
\bE=\omega\bE^{(1)}+\Ocal(\omega^2).
\eeq
By using \eqnref{eq:leadd0104}, \eqnref{eq:leadd0105} and the transmission condition \eqnref{eq:addinnertran1} again, it is readily seen that $\bE^{(1)}$ satisfies
\beq\label{eq:leadd0107}
\left\{
\begin{array}{ll}
\nabla\times\bE^{(1)}=i\mu \bH^{(0)} & \mbox{in} \quad B,\\
 \nabla\cdot\Gs\bE^{(1)}=0 &\mbox{in} \quad B\\
 \nu\cdot\Gs\bE^{(1)}=-i\epsilon_0\nu\cdot\bE|_+ &\mbox{on} \quad \p B.
\end{array}
\right.
\eeq
From the proof of Lemma \ref{le:01} one can see that \eqnref{eq:leadd0107} has a unique solution, and straightforward verifications show that $\bE^{(1)}=\nabla u +\nabla\times\Psi$ is exactly the solution to \eqnref{eq:leadd0107}.

The proof is complete.
\end{proof}

\begin{defn}\label{df:03}
Suppose $\mu\neq \mu_0$ is a constant in $B$ and $\Gs$, $\epsilon$ are piecewise constants in $B$ which satisfy \eqnref{eq:multilayered} in Assumption \ref{asm:01} with $N=2$ and $\Gs\neq 0$. Let $(\bE, \bH)$ be the corresponding solutions to \eqnref{eq:pss} and \eqnref{eq:radia2}, associated with $(\epsilon, \mu, \Gs, \bJ)$.
Let $\bE^{(1)}$ be defined in \eqnref{eq:leadd0107} associated with $(\epsilon, \mu, \Gs, \bJ)$. We call $(\bJ, \mu, \Gs)$ an admissible two-layer structure, if there exists $l_1, l_2\in L_0^2(\p \Sigma_1)$, such that
\beq\label{eq:dfadd0101}
\begin{split}
&\int_{\p \Sigma_1} l_1 \nu\cdot\bE^{(1)}|_+=C_1\int_{\p B} g_{l_1}\nu\cdot\bE^{(1)}|_-,\\
&\int_{\p \Sigma_1} l_2 \nu\cdot\bE^{(1)}|_+=C_2\int_{\p B} g_{l_2}\nu\cdot\bE^{(1)}|_-,
\end{split}
\eeq
where $C_1\neq C_2$ are two constants and $g_{l_j}=u_j|_{\p B}$, with $u_j$, $j=1, 2$ solutions to
\beq\label{eq:dfadd0102}
\left\{
\begin{split}
&\Delta u_j=0\quad \mbox{in}\quad B, \\
&u_j=-\Scal_{B}^0(\Scal_B^0)^{-1}\left[\Scal_{\Sigma_1}^0(\Scal_{\Sigma_1}^0)^{-1}(\Gs_1^{(1)})^{-1}\Gs_1^{(2)}
\Big(\frac{I}{2}+\Kcal_{\Sigma_1}^0\Big)[l_j]-\mathcal{D}_{\Sigma_1}^0[l_j]\right]\\
&\quad\quad+(\Gs_1^{(1)})^{-1}\Gs_1^{(2)}\Big(\frac{I}{2}+\Kcal_{\Sigma_1}^0\Big)[l_j]+\Big(\frac{I}{2}-\Kcal_B^0\Big)[l_j] \quad \mbox{on} \quad \p \Sigma_1.
\end{split}
\right.
\eeq
\end{defn}

\begin{thm}\label{th:unique2}
Let $(\epsilon_1, \mu_1, \Gs_1,\bJ_1)$ and $(\epsilon_2, \mu_2, \Gs_2, \bJ_2)$ be two sets of EM configurations. Suppose that ($\epsilon_j$, $\mu_j$, $\Gs_j$) and $\bJ_j\in H_0(\mathrm{div}; B)$ verify \eqnref{eq:outbd1}, $j=1,2$. Suppose further that $\mu_1=\mu_2=\mu'$  and $\nabla\cdot\bJ_1=\nabla\cdot\bJ_2=0$ in $B$, where $\mu'$ is a positive constant, and $\Gs_j$, $\epsilon_j$ are piecewise constants in $B$ which satisfy \eqnref{eq:multilayered} in Assumption \ref{asm:01} with $\Gs_1=\Gs_2\neq 0$. Let $(\bE_1, \bH_1)$ and $(\bE_2, \bH_2)$ be the corresponding solutions to \eqnref{eq:pss} and \eqnref{eq:radia2}, associated with $(\epsilon_1, \mu_1, \Gs_1, \bJ_1)$ and $(\epsilon_2, \mu_2, \Gs_2, \bJ_2)$, respectively.
Suppose that
\begin{equation}\label{eq:th2c2}
\Pi_{\epsilon_1,\mu_1,\Gs_1,\bJ_1}(\Bx, \omega)=\Pi_{\epsilon_2, \mu_2, \Gs_2, \bJ_2}(\Bx, \omega), \quad (\Bx, \omega)\in \partial B\times (0, \omega_0),
\end{equation}
where $\omega_0$ is any given positive constant.
Then we have $\bJ_1=\bJ_2$, provided there exists $\xi\in \mathbb{S}^2$ such that
$\Phi:=\xi\cdot(\nabla\times (\bJ_1-\bJ_2))$ satisfies the {admissibility} condition in Definition \ref{df:02}.
Furthermore, suppose $N=2$ in \eqnref{eq:multilayered}, let $\bE^{(1)}$ be the solution to \eqnref{eq:leadd0107}. If $\nu\cdot\bE|_+\neq 0$ as $\omega\rightarrow 0$ on $\p B$ and $(\bJ_1, \mu', \Gs_1)$ is an admissible two-layer structure, then $\epsilon_1=\epsilon_2$.
\end{thm}
\begin{proof}
By following a similar proof in the second case of Theorem \ref{th:unique1}, one can show $\bJ_1=\bJ_2$ and $\bH_1-\bH_2=\Ocal(\omega)$ in $B$.
By Lemma \ref{le:lead01} one can expand $\bE_j$ and $\bH_j$ in $B$, $j=1, 2$, by
\beq\label{eq:rev040101}
\bE_j=\bE^{(1)}\omega+\bE_j^{(2)}\omega^2+\Ocal(\omega^3), \quad \bH_j=\bH^{(0)}+\bH_1^{(1)}\omega+\Ocal(\omega^2),
\eeq
where $\bH^{(0)}$ is defined in \eqnref{eq:leadingodH01} and
\beq\label{eq:rev040102}
\bE^{(1)}=\nabla u+\nabla\times\Psi,
\eeq
where $u$ and $\Psi$ are defined in \eqnref{eq:leadd0101} and \eqnref{eq:leadd010101}.
Then one has
\beq\label{eq:rev040103}
\nabla\times(\bE_1-\bE_2)=i\omega\mu'(\bH_1-\bH_2)=\Ocal(\omega^2), \quad \mbox{in} \quad B,
\eeq
and thus
\beq\label{eq:rev040104}
\nabla\times\nabla\times(\bH_1-\bH_2)=-i\omega\nabla\times(\epsilon_1\bE_1-\epsilon_2\bE_2)+\Gs\nabla\times(\bE_1-\bE_2)=\Ocal(\omega^2).
\eeq
Note that $\nabla\times\nabla=-\Delta+\nabla\nabla\cdot$ and $\nabla\cdot(\bH_1-\bH_2)=0$ in $B$, \eqnref{eq:rev040104} implies
\beq\label{eq:rev040105}
\Delta(\bH_1-\bH_2)=\Ocal(\omega^2).
\eeq
Combining \eqref{eq:rev040105} and the facts $\nu\times(\bH_1-\bH_2)=0$ and $\nu\cdot(\bH_1-\bH_2)=0$ on $\p B$, one can easily obtain that
$\bH_1-\bH_2=\Ocal(\omega^2)$. From \eqnref{eq:rev040103} one thus has $\nabla\times(\bE_1-\bE_2)=\Ocal(\omega^3)$, that is,
$\nabla\times(\bE_1^{(2)}-\bE_2^{(2)})=0$ in $B$. Since $\nabla\cdot\bJ=0$, there hold $\nabla\cdot\bE_j=0$ in $B\setminus\overline{\Sigma_1}$ and $\Sigma_1$, $j=1, 2$.
Next, by using the transmission condition \eqnref{eq:addinnertran1} one has
\beq\label{eq:rev040106}
\begin{split}
\nu\cdot(\epsilon_j^{(1)}+i\Gs_j^{(1)}\omega^{-1})\bE_j|_-&=\nu\cdot\epsilon_0\bE_j|_+\quad \mbox{on} \quad \p B,\\
\nu\cdot(\epsilon_j^{(2)}+i\Gs_j^{(2)}\omega^{-1})\bE_j|_-&=\nu\cdot(\epsilon_j^{(1)}+i\Gs_j^{(1)}\omega^{-1})\bE_j|_+ \quad \mbox{on} \quad \p \Sigma_1,
\end{split}
\eeq
Noting that $\Gs_1=\Gs_2$, one thus has
\beq\label{eq:rev040107}
\begin{split}
& i\Gs_1^{(1)}\nu\cdot(\bE_1^{(2)}-\bE_2^{(2)})|_-=-(\epsilon_1^{(1)}-\epsilon_2^{(1)})\nu\cdot\bE^{(1)}|_- \quad \mbox{on} \quad \p B,\\
& i\Gs_1^{(1)}\nu\cdot(\bE_1^{(2)}-\bE_2^{(2)})|_++(\epsilon_1^{(1)}-\epsilon_2^{(1)})\nu\cdot\bE^{(1)}|_+ \\
=& i\Gs_1^{(2)}\nu\cdot(\bE_1^{(2)}-\bE_2^{(2)})|_-+(\epsilon_1^{(2)}-\epsilon_2^{(2)})\nu\cdot\bE^{(1)}|_- \quad \mbox{on} \quad \p \Sigma_1.
\end{split}
\eeq
Define $\bE_1^{(2)}-\bE_2^{(2)}:=\nabla \tilde u$, then $\tilde u$ has the following form
\beq\label{eq:rev040108}
\tilde{u}=\left\{
\begin{split}
&\Scal_B^0[\phi_1]+\Scal_{\Sigma_1}^0[\phi_2] \quad \mbox{in} \quad B\setminus\overline{\Sigma_1},\\
&\Scal_{\Sigma_1}^0[\phi_3]\quad \mbox{in} \quad \Sigma_1,
\end{split}
\right.
\eeq
where $\phi_1$, $\phi_2$ and $\phi_3$ satisfy
\beq\label{eq:rev040109}
\left\{
\begin{split}
\Big(\frac{I}{2}+(\Kcal_B^0)^*\Big)[\phi_1]+\frac{\p}{\p \nu}\Scal_{\Sigma_1}^0[\phi_2]=\nu\cdot(\bE_1^{(2)}-\bE_2^{(2)})|_- &\quad \mbox{on} \quad \p B, \\
\frac{\p}{\p \nu}\Scal_{B}^0[\phi_1]+\Big(-\frac{I}{2}+(\Kcal_{\Sigma_1}^0)^*\Big)[\phi_2]=\nu\cdot(\bE_1^{(2)}-\bE_2^{(2)})|_+ &\quad \mbox{on} \quad \p \Sigma_1, \\
\Big(\frac{I}{2}+(\Kcal_{\Sigma_1}^0)^*\Big)[\phi_3]=\nu\cdot(\bE_1^{(2)}-\bE_2^{(2)})|_- &\quad \mbox{on} \quad \p \Sigma_1,
\end{split}
\right.
\eeq
and
\beq\label{eq:rev040110}
\left\{
\begin{split}
\Scal_B^0[\phi_1]+\Scal_{\Sigma_1}^0[\phi_2]=0 &\quad \mbox{on} \quad \p B, \\
\Scal_B^0[\phi_1]+\Scal_{\Sigma_1}^0[\phi_2]=\Scal_{\Sigma_1}^0[\phi_3] &\quad \mbox{on} \quad \p \Sigma_1.
\end{split}
\right.
\eeq
For notational convenience, we define
\beq\label{eq:rev040111}
\begin{split}
&f_1:=\nu\cdot(\bE_1^{(2)}-\bE_2^{(2)}), \quad e_1:=(\epsilon_1^{(1)}-\epsilon_2^{(1)})\nu\cdot\bE^{(1)}|_-,\quad \mbox{on} \quad \p B,\\
&f_2:=\nu\cdot(\bE_1^{(2)}-\bE_2^{(2)})|_+, \quad e_2:=(\epsilon_1^{(1)}-\epsilon_2^{(1)})\nu\cdot\bE^{(1)}|_+, \quad \mbox{on} \quad \p \Sigma_1,\\
&f_3:=\nu\cdot(\bE_1^{(2)}-\bE_2^{(2)})|_-,  \quad e_3:=(\epsilon_1^{(2)}-\epsilon_2^{(2)})\nu\cdot\bE^{(1)}|_-,\quad \mbox{on} \quad \p \Sigma_1.
\end{split}
\eeq
Let $\mathcal{H}=L^2(\p B)\times L^2(\p \Sigma_1)$ and the Neumann-Poincar\'e-type operator $\mathbb{K}^*: \mathcal{H}\rightarrow \mathcal{H}$ be
\beq\label{eq:rev040112}
\mathbb{K}^*:=
\left[
\begin{array}{cc}
-(\Kcal_B^0)^* & -\frac{\p}{\p \nu}\Scal_{\Sigma_1}^0 \\
\frac{\p}{\p \nu}\Scal_{B}^0 & (\Kcal_{\Sigma_1}^0)^*
\end{array}
\right]
\eeq
It is shown in \cite{ACKL14} that the $L^2$-adjoint of $\mathbb{K}^*$, namely $\mathbb{K}$, is given by
\beq\label{eq:rev040113}
\mathbb{K}:=
\left[
\begin{array}{cc}
-\Kcal_B^0 & \mathcal{D}_{\Sigma_1}^0 \\
-\mathcal{D}_{B}^0 & \Kcal_{\Sigma_1}^0
\end{array}
\right],
\eeq
where $\mathcal{D}_{\Sigma_1}^0$ and $\mathcal{D}_{B}^0$ are the double layer potential operators defined on $\p \Sigma_1$ and $\p B$, respectively.
Then the first two equations in \eqnref{eq:rev040109} can be rewritten in the form
\beq\label{eq:rev040114}
\Big(-\frac{1}{2}
\mathbb{I}+\mathbb{K}^*\Big)[\mathbf{p}]=\mathbf{f},
\eeq
where $\mathbf{p}:=(\phi_1, \phi_2)^T$, $\mathbf{f}:=(-f_1, f_2)^T$, and $\mathbb{I}$ is the identity operator on $\mathcal{H}$.
Define
\beq\label{eq:rev040115}
\mathbb{S}:=
\left[
\begin{array}{cc}
\Scal_B^0 & \Scal_{\Sigma_1}^0\\
\Scal_B^0 & \Scal_{\Sigma_1}^0
\end{array}
\right],
\eeq
then there holds the Calder\'on's identity $\mathbb{S}\mathbb{K}^*=\mathbb{K}\mathbb{S}$ (see \cite{ACKL14}). By applying $\mathbb{S}$ on both sides of  \eqnref{eq:rev040114} and using the Calder\'on's identity, one thus has
\beq\label{eq:rev040116}
\Big(-\frac{1}{2}
\mathbb{I}+\mathbb{K}\Big)\mathbb{S}[\mathbf{p}]=\mathbb{S}[\mathbf{f}].
\eeq
Similarly, by applying $\Scal_{\Sigma_1}^0$ on both sides of the third equation in \eqnref{eq:rev040109} one has
\beq\label{eq:rev040117}
\Big(\frac{I}{2}+\Kcal_{\Sigma_1}^0\Big)\Scal_{\Sigma_1}^0[\phi_3]=\Scal_{\Sigma_1}^0[f_3].
\eeq
By \eqnref{eq:rev040110}, one additionally has
\beq\label{eq:rev040118}
\mathbb{S}[\mathbf{p}]=(0, \Scal_{\Sigma_1}^0[\phi_3])^T.
\eeq
%By taking boundary integral on both sides of \eqnref{eq:rev040116} and note that $\mathbb{K}[(1, 1)^T]=(1/2, 1/2)^T$ one has
%$$
%\int_{\p B\cup \p \Sigma_1}\mathbb{S}[\mathbf{f}]=0,
%$$
%and thus
%\beq\label{eq:revadd0101}
%\int_{\p \Sigma_1}\Scal_{\Sigma_1}^0[\phi_3]=0.
%\eeq
Combing \eqnref{eq:rev040116}-\eqnref{eq:rev040118} and \eqnref{eq:rev040107}, one can derive that
\beq\label{eq:rev040119}
\mathbb{S}^{-1}\Big(-\frac{1}{2}
\mathbb{I}+\mathbb{K}\Big)\mathbb{S}[\mathbf{p}]=(\Gs_1^{(1)})^{-1}\Gs_1^{(2)}\mathbf{q}+(i\Gs_1^{(1)})^{-1}\mathbf{e},
\eeq
where $\mathbf{q}=(0, (\Scal_{\Sigma_1}^0)^{-1}\Big(\frac{I}{2}+\Kcal_{\Sigma_1}^0\Big)\Scal_{\Sigma_1}^0[\phi_3])^T$, and
$\mathbf{e}=(e_1, e_3-e_2)^T$.

Suppose $l\in L_0^2(\p \Sigma_1)$, and define $m:=(\Gs_1^{(1)})^{-1}\Gs_1^{(2)}(\Scal_{\Sigma_1}^0)^{-1}\Big(\frac{I}{2}+\Kcal_{\Sigma_1}^0\Big)[l]$. Let $g, h\in L^2(\p B)$ be the solution to
\beq\label{eq:rev040120}
\left\{
\begin{split}
-\Big(\frac{I}{2}+\Kcal_B^0\Big)[g]-\Scal_B^0[h]=\Scal_{\Sigma_1}^0[m]-\mathcal{D}_{\Sigma_1}^0[l] &\quad \mbox{on} \quad \p B,\\
-\mathcal{D}_B^0[g]-\Scal_B^0[h]=\Scal_{\Sigma_1}^0[m]+\Big(\frac{I}{2}-\Kcal_{\Sigma_1}^0\Big)[l] &\quad \mbox{on} \quad \p \Sigma_1.
\end{split}
\right.
\eeq
We remark that $g$ and $h$ are uniquely solvable. In fact, from the first equation in \eqnref{eq:rev040120}, we have
\beq\label{eq:rev040121}
h=-(\Scal_{B}^0)^{-1}[h_1]-(\Scal_B^0)^{-1}\Big(\frac{I}{2}+\Kcal_B^0\Big)[g].
\eeq
where $h_1:=\Scal_{\Sigma_1}^0[m]-\mathcal{D}_{\Sigma_1}^0[l]$ on $\p B$.
By substituting \eqnref{eq:rev040121} into the second equation in \eqnref{eq:rev040120}, one has
\beq\label{eq:rev040122}
\left(-\mathcal{D}_B^0+\Scal_B^0(\Scal_B^0)^{-1}\Big(\frac{I}{2}+\Kcal_B^0\Big)\right)[g]=h_2 \quad \mbox{on} \quad \p \Sigma_1,
\eeq
where $$h_2:=\Scal_{\Sigma_1}^0[m]+\Big(\frac{I}{2}-\Kcal_{\Sigma_1}^0\Big)[l]-\Scal_{B}^0(\Scal_B^0)^{-1}[h_1] \quad \mbox{on} \quad \p \Sigma_1.$$
Note that the left side of \eqnref{eq:rev040122} can be extended uniquely to a harmonic function in $B$. By taking the trace on $\p B$ and using the jump formula one thus has
\beq\label{eq:rev040123}
g=u_c|_{\p B} \quad \mbox{on} \quad \p B,
\eeq
where $u_c$ is a harmonic function in $B$ and $u_c=h_2$ on $\p \Sigma_1$.
From \eqnref{eq:rev040120}, one can derive that
\beq\label{eq:rev040124}
\mathbb{S}^{-1}\Big(-\frac{1}{2}
\mathbb{I}+\mathbb{K}\Big)\left[
\begin{array}{c}
g \\
l
\end{array}
\right]
=\left[\begin{array}{c}
h \\
m
\end{array}
\right].
\eeq
Hence by taking the inner product in $\mathcal{H}$ on both sides of \eqnref{eq:rev040119} with $(g, l)^T$, one finally obtains
\beq\label{eq:rev040125}
\begin{split}
\int_{\p \Sigma_1}m\Scal_{\Sigma_1}^0[\phi_3]=&(\Gs_1^{(1)})^{-1}\Gs_1^{(2)}\int_{\p \Sigma_1}(\Scal_{\Sigma_1}^0)^{-1}\Big(\frac{I}{2}+\Kcal_{\Sigma_1}^0\Big)\Scal_{\Sigma_1}^0[\phi_3]l\\
&+(i\Gs_1^{(1)})^{-1}\Big(\int_{\p B}ge_1+\int_{\p \Sigma_1}l(e_3-e_2)\Big).
\end{split}
\eeq
By the definition of $m$, we thus have
\beq\label{eq:rev040126}
\int_{\p B}ge_1+\int_{\p \Sigma_1}l(e_3-e_2)=0,
\eeq
or equivalently
\beq\label{eq:rev040127}
\begin{split}
&(\epsilon_1^{(1)}-\epsilon_2^{(1)})\Big(\int_{\p B}g\nu\cdot\bE^{(1)}|_--\int_{\p \Sigma_1}l\nu\cdot\bE^{(1)}|_+\Big)\\
&+(\epsilon_1^{(2)}-\epsilon_2^{(2)})\int_{\p \Sigma_1}l\nu\cdot\bE^{(1)}|_-=0.
\end{split}
\eeq
Define $t_1:=(\epsilon_1^{(1)}-\epsilon_2^{(1)})$ and $t_2:=(\epsilon_1^{(2)}-\epsilon_2^{(2)})$. Since $(\bJ_1, \mu', \Gs_1)$ is an admissible two layer structure, by \eqnref{eq:dfadd0101} in Definition \ref{df:03}, one immediately has
\beq\label{eq:rev040128}
\begin{split}
&t_1(C_1-1)\int_{\p \Sigma_1}l_1\nu\cdot\bE^{(1)}|_++t_2\int_{\p \Sigma_1}l_1\nu\cdot\bE^{(1)}|_-=0,\\
&t_1(C_2-1)\int_{\p \Sigma_1}l_2\nu\cdot\bE^{(1)}|_++t_2\int_{\p \Sigma_1}l_2\nu\cdot\bE^{(1)}|_-=0.
\end{split}
\eeq
Finally, by using the relation $\Gs_1^{(1)}\nu\cdot\bE^{(1)}|_+=\Gs_1^{(2)}\nu\cdot\bE^{(1)}|_-$ on $\p \Sigma_1$, one can easily find that the equations \eqnref{eq:rev040128} with respect to $t_1$ and $t_2$ have unique solutions $t_1=0$ and $t_2=0$.

The proof is complete.
\end{proof}

\begin{rem}
In Theorem~\ref{th:unique2}, we show that if the source satisfies the admissibility condition and the EM medium parameters are piecewise constants, then the source term can be uniquely recovered by the boundary EM measurement. In the case that the EM inhomogeneity is of a two-layer structure and satisfies the admissibility condition \eqref{eq:dfadd0101}, then the EM medium parameters can also be recovered. We would like to remark the argument in the proof of Theorem~\ref{th:unique2} can be extended to deriving a similar unique recovery result for a general $N$-layer piecewise constant medium. In fact, by constructing $N-1$ equations similar to \eqnref{eq:rev040114} and then by some tedious computations, one should be able to derive a similar equation to \eqnref{eq:rev040127}. In such a case, one shall need to impose some similar assumptions to Definition \ref{df:03} on the EM configuration. The expression of those assumptions are too lengthy and hence we only present the result for the two-layer case. On the other hand, it is emphasized that we believe that the condition \eqref{eq:dfadd0101} is a generic one and this can be verified in a particular case within the spherical geometry.   

\end{rem}

\section*{Acknowledgment}
The work of Y. Deng was supported by NSF grant of China, No. 11601528, NSF grant of Hunan No. 2017JJ3432, Mathematics and Interdisciplinary Sciences Project, Central South University. The work of H. Liu was supported by the FRG and startup grants from Hong Kong Baptist University, Hong Kong RGC General Research Funds, 12302415 and 12302017. G. Uhlmann was partly supported by NSF.

\end{document}